\documentclass[12pt,british,svgnames]{article}

\usepackage{babel}
\usepackage{mathrsfs} 
\usepackage{amsmath,amsfonts,amssymb,amsthm,enumerate}
\usepackage{ucs} % Unicode support
\usepackage[utf8x]{inputenc} % UCS' UTF-8 driver is better than the LaTeX kernel's
\usepackage[T1]{fontenc} % The default font encoding only contains Latin characters
\usepackage{textcomp,url}
\usepackage{eucal}
\usepackage{comment}
\usepackage[noBBpl]{mathpazo}
\usepackage[pdftex,                %
implicit=true,
bookmarks=true,
breaklinks=true,
bookmarksnumbered = true,%     % Signets numérotés
colorlinks= true,%     % Liens en couleur
urlcolor= blue,
anchorcolor = yellow,
citecolor=blue,%  % Couleur des liens externes
pdfborder= {0 0 0}%   % Style de bordure : ici, pas de bordure
]{hyperref}%                   % Utilization de HyperTeX

\usepackage{multicol}%Oscar: I introduce this package to use in group presentations

\usepackage{tikz}
\usepackage{braids}
\usepackage{subcaption}

\makeatletter

\renewcommand{\p@enumii}{}
\def\@enum@{\list{\csname label\@enumctr\endcsname}%
{\usecounter{\@enumctr}\def\makelabel##1{
\normalfont\ignorespaces\emph{{##1}~}}
\setlength{\labelsep}{3pt}
\setlength{\parsep}{0pt}
\setlength{\itemsep}{0pt}
\setlength{\leftmargin}{0pt}
\setlength{\labelwidth}{0pt}
\setlength{\listparindent}{\parindent}
\setlength{\itemsep}{0pt}
\setlength{\itemindent}{0pt}
\topsep=3pt plus 1pt minus 1 pt}}
\makeatother

\renewcommand{\epsilon}{\ensuremath{\varepsilon}}
\renewcommand{\phi}{\ensuremath{\varphi}}

\renewcommand{\to}{\ensuremath{\longrightarrow}}
\renewcommand{\mapsto}{\ensuremath{\longmapsto}}

\newcommand{\N}{\ensuremath{\mathbb N}}
\newcommand{\Z}{\ensuremath{\mathbb Z}}

\newcommand{\sn}[1][n]{\ensuremath{S_{{#1}}}}

\renewcommand{\ker}[1]{\ensuremath{\operatorname{\text{Ker}}\left({#1}\right)}}

\newcommand{\aut}[1]{\ensuremath{\operatorname{\text{Aut}}\left({#1}\right)}}

\newcommand{\id}{\ensuremath{\operatorname{\text{Id}}}}
\newcommand{\lhra}{\lhook\joinrel\longrightarrow}

\makeatletter
\def\@map#1#2[#3]{\mbox{$#1 \colon\thinspace #2 \to #3$}}
\def\map#1#2{\@ifnextchar [{\@map{#1}{#2}}{\@map{#1}{#2}[#2]}}
\makeatother

\newcommand{\brak}[1]{\ensuremath{\left\{ #1 \right\}}}
\newcommand{\ang}[1]{\ensuremath{\left\langle #1\right\rangle}}

\newcommand{\setangr}[2]{\ensuremath{\ang{#1 \,\left\lvert \, #2 \right.}}}

\newcommand{\ord}[1]{\ensuremath{\left\lvert #1\right\rvert}}

\newcommand{\setl}[2]{\ensuremath{\brak{\left. #1 \,\right\rvert \, #2}}}

\newtheoremstyle{theoremm}{}{}{\itshape}{}{\scshape}{.}{ }{}

\theoremstyle{theoremm}
\newtheorem{thm}{Theorem}
\newtheorem{lem}[thm]{Lemma}
\newtheorem{prop}[thm]{Proposition}
\newtheorem{cor}[thm]{Corollary}

\newtheoremstyle{remark}{}{}{}{}{\scshape}{.}{ }{}

\theoremstyle{remark}

\newtheorem{rem}[thm]{Remark}
\newtheorem{rems}[thm]{Remarks}

\newtheorem{exms}[thm]{Examples}

\newtheoremstyle{comment}{}{}{\bfseries}{}{\bfseries}{:}{ }{}

\theoremstyle{comment}

\newcommand{\reth}[1]{Theorem~\protect\ref{th:#1}}
\newcommand{\relem}[1]{Lemma~\protect\ref{lem:#1}}
\newcommand{\repr}[1]{Proposition~\protect\ref{prop:#1}}
\newcommand{\reco}[1]{Corollary~\protect\ref{cor:#1}}
\newcommand{\rexs}[1]{Examples~\protect\ref{ex:#1}}
\newcommand{\resec}[1]{Section~\protect\ref{sec:#1}}

\newcommand{\req}[1]{equation~(\protect\ref{eq:#1})}
\newcommand{\reqref}[1]{(\protect\ref{eq:#1})}

\begin{document}

\title{Embeddings of finite groups in $B_n/\Gamma_k(P_n)$ for $k=2,3$} 

\author{DACIBERG~LIMA~GON\c{C}ALVES\\
Departamento de Matem\'atica - IME-USP,\\
Rua~do~Mat\~ao~1010~CEP:~05508-090 - S\~ao Paulo - SP - Brazil.\\
e-mail:~\url{dlgoncal@ime.usp.br}\vspace*{4mm}\\
JOHN~GUASCHI\\
Normandie Univ., UNICAEN, CNRS,\\
Laboratoire de Math\'ematiques Nicolas Oresme UMR CNRS~\textup{6139},\\
CS 14032, 14032 Cedex Cedex 5, France.\\%
e-mail:~\url{john.guaschi@unicaen.fr}\vspace*{4mm}\\
OSCAR~OCAMPO~\\
Universidade Federal da Bahia, Departamento de Matem\'atica - IME,\\
Av. Adhemar de Barros~S/N~CEP:~40170-110 - Salvador - BA - Brazil.\\
e-mail:~\url{oscaro@ufba.br}
}

%\date{}

\maketitle

\begingroup%Localising the change to `thefootnote'.
\renewcommand{\thefootnote}{}%Removing the footnote symbol.
\footnotetext{\hspace*{-7.3mm}
 2010 AMS Subject Classification: 20F36 Braid groups, Artin groups;
20F18 Nilpotent groups;
20F14 Derived series, central series, and generalizations;
20B35 Subgroups of symmetric groups;
16S34 Group rings.}
\endgroup

\begin{abstract}
\noindent
\emph{Let $n\geq 3$. 
%\comj{Rewrite this, once the other corrections have been made.} 
In this paper, we study the problem of whether a given finite group $G$ embeds in a quotient of the form $B_{n}/\Gamma_{k}(P_{n})$, where $B_{n}$ is the  $n$-string Artin braid group, $k\in \brak{2,3}$, and $\brak{\Gamma_{l}(P_{n})}_{l\in \N}$ is the lower central series of the $n$-string pure braid group $P_{n}$. Previous results show that a necessary condition for such an embedding to exist is that $\ord{G}$ is odd (resp.\ is relatively prime with $6$) if $k=2$ (resp.\ $k=3$). We show that any finite group $G$ of odd order (resp.\ of order relatively prime with $6$) embeds in $B_{\ord{G}}/\Gamma_2(P_{\ord{G}})$ (resp.\ in $B_{\ord{G}}/\Gamma_3(P_{\ord{G}})$), where $\ord{G}$ denotes the order of $G$. The result in the case of $B_{\ord{G}}/\Gamma_2(P_{\ord{G}})$ has been proved independently by Beck and Marin. One may then ask whether $G$ embeds in a quotient of the form $B_{n}/\Gamma_{k}(P_{n})$, where $n<\ord{G}$ and $k\in \brak{2,3}$. If $G$ is of the form $\Z_{p^r}\rtimes_{\theta} \Z_d$, where the action $\theta$ is injective, $p$ is an odd prime (resp.\ $p\geq 5$ is prime) and $d$ is odd (resp.\ $d$ is relatively prime with $6$) and divides $p-1$, we show that $G$ embeds in $B_{p^r}/\Gamma_2(P_{p^r})$ (resp.\ in $B_{p^r}/\Gamma_3(P_{p^r})$). In the case $k=2$, this extends a result of Marin concerning the embedding of the Frobenius groups in $B_{n}/\Gamma_{2}(P_{n})$, and is a special case of another result of Beck and Marin. Finally, we construct an explicit embedding in $B_9/\Gamma_2(P_9)$ of the two non-Abelian groups of order $27$, namely the semi-direct product $\Z_9\rtimes \Z_3$, where the action is given by multiplication by $4$, and the Heisenberg group mod $3$.}
% and  $d$ relatively prime to $p$ we show that if $\ord{G}$ is odd then it  embeds in $B_{p^r}/\Gamma_2(P_{p^r})$ and if $\ord{G}$  is relatively prime with $6$ then it also embeds in $B_{p^r}/\Gamma_3(P_{p^r})$. Also we provide an embedding of $\Z_9\rtimes \Z_3$,  
%% in $B_9/\Gamma_2(P_9)$ 
%where the action is multiplication by $4$, and an embedding  of the $mod(3)$ Heisenberg group,   in $B_9/\Gamma_2(P_9)$.}
\end{abstract}

\section{Introduction}\label{sec:intro}

If $n\in \N$, let $B_{n}$ denote the \emph{(Artin) braid group} on $n$ strings. It is well known that $B_{n}$ admits a presentation with generators $\sigma_{1},\ldots,\sigma_{n-1}$ that are subject to the relations $\sigma_{i}\sigma_{j}=\sigma_{j}\sigma_{i}$ for all $1\leq i<j\leq n-1$ for which $\lvert i-j \rvert\geq 2$, and $\sigma_{i}\sigma_{i+1}\sigma_{i}=\sigma_{i+1}\sigma_{i}\sigma_{i+1}$ for all $1\leq i\leq n-2$. Let $\map{\sigma}{B_{n}}[\sn]$ denote the surjective homomorphism onto the symmetric group $\sn$ defined by $\sigma(\sigma_{i})=(i,i+1)$ for all $1\leq i\leq n-1$. The \emph{pure braid group} $P_{n}$ on $n$ strings is defined to be the kernel of $\sigma$, from which we obtain the following short exact sequence:
\begin{equation}\label{eq:sespn}
1 \to P_n \to  B_n \stackrel{\sigma}{\to} \sn \to 1. 
\end{equation}
If $G$ is a group, recall that its \emph{lower central series} $\brak{\Gamma_{k}(G)}_{k\in \N}$ is defined by $\Gamma_{1}(G)=G$, and $\Gamma_{k}(G)=[\Gamma_{k-1}(G),G]$ for all $k\geq 2$ (if $H$ and $K$ are subgroups of $G$, $[H,K]$ is defined to be the subgroup of $G$ generated by the commutators of the form $[h,k]=hkh^{-1}k^{-1}$, where $h\in H$ and $k\in K$). Note that $\Gamma_{2}(G)$ is the commutator subgroup of $G$, and that $\Gamma_{k}(G)$ is a normal subgroup of $G$ for all $k\in \N$. In our setting, since $P_{n}$ is normal in $B_{n}$, it follows that $\Gamma_{k}(P_{n})$ is also normal in $B_{n}$, and the extension~\reqref{sespn} induces the following short exact sequence:
\begin{equation}\label{eq:sesgammak}
1 \to P_n/\Gamma_k(P_n) \to  B_n/\Gamma_k(P_n) \stackrel{\overline{\sigma}}{\to} \sn \to 1,
\end{equation}
obtained by taking the quotient of $P_n$ and $B_{n}$ by $\Gamma_k(P_n)$. It follows from results of Falk and Randell~\cite{FRinv} and Kohno~\cite{Ko} that the kernel of~\reqref{sesgammak} is torsion free (see \repr{tf} for more information). The quotient groups of the form $B_{n}/\Gamma_{k}(P_{n})$ have been the focus of several recent papers. First, the quotient $B_n/\Gamma_{2}(P_{n})$ belongs to a family of groups known as \emph{enhanced symmetric groups}~\cite[page~201]{Marin2} that were analysed in~\cite{Tits}. Secondly, in their study of pseudo-symmetric braided categories, Panaite and Staic showed that this quotient is isomorphic to the quotient of $B_n$ by the normal closure of the set $\setl{\sigma_i\sigma_{i+1}^{-1}\sigma_i \sigma_{i+1}^{-1} \sigma_i \sigma_{i+1}^{-1}}{i=1,2,\ldots,n-2}$~\cite{PS}. 
%subgroup generated by the relations $\sigma_i\sigma_{i+1}^{-1}\sigma_i=\sigma_{i+1}\sigma_i^{-1}\sigma_{i+1}$ for $i=1,2,\ldots,n-2$
%It also arises in the study of pseudo-symmetric braided categories by Panaite and Staic. They consider the quotient, denoted by $PS_n$, of $B_n$ by the normal subgroup generated by the relations $\sigma_i\sigma_{i+1}^{-1}\sigma_i=\sigma_{i+1}\sigma_i^{-1}\sigma_{i+1}$ for $i=1,2,\ldots,n-2$, and they show that it is isomorphic to $B_n/[P_n,P_n]$~\cite{PS}.
Thirdly, in~\cite{GGO1}, we showed that $B_n/\Gamma_2(P_n)$ is a crystallographic group, and that up to isomorphism, its finite Abelian subgroups are the Abelian subgroups of $\sn$ of odd order. In particular, the torsion of $B_n/\Gamma_2(P_n)$ is the odd torsion of $\sn$. We also gave an explicit embedding in $B_7/\Gamma_2(P_7)$ of the Frobenius group $\Z_7\rtimes \Z_3$ of order $21$, which is the smallest finite non-Abelian group of odd order. As far as we know, this is the first example of a finite non-Abelian group that embeds in a quotient of the form $B_n/\Gamma_2(P_n)$.
% and its torsion was characterised, and also it  was proved that the Frobenius group $\Z_7\rtimes \Z_3$ embeds in $B_7/\Gamma_2(P_7)$ where an explict  embedding was given, see \cite[Theorem~7]{GGO1}. To the best of our knowledge,  this was the first example of a finite non-Abelian group that embeds in $B_n/\Gamma_2(P_n)$.
 % and also the embedding was explicitly given.
%, namely it was shown that the Frobenius group $\Z_7\rtimes \Z_3$ of order 21 embeds  in  $B_7/\Gamma_2(P_7)$. 
Almost all of the results of~\cite{GGO1} were subsequently extended to the generalised braid groups associated to an arbitrary complex reflection group by Marin~\cite{Ma}. If $p>3$ is a prime number for which $p\equiv 3\bmod{4}$, he showed that the Frobenius group $\Z_p\rtimes \Z_{(p-1)/2}$ embeds in $B_p/\Gamma_2(P_p)$. 
%have since been extended in different directions. 
%In \cite{Ma}, Marin extended the results given in \cite{GGO1} to the generalized braid group associated to an arbitrary complex reflection group, and also he shows that   the Frobenius group $\Z_p\rtimes \Z_{(p-1)/2}$ with 
%$p>3$ a prime number such that  $p \equiv 3$ mod $4$,  embeds in $B_p/\Gamma_2(P_p)$.  
Observe that this group cannot be embedded in $B_n/\Gamma_2(P_n)$ for any $n<p$, since $\Z_p$ cannot be embedded in $\sn$ in this case. In another direction, the authors studied some aspects of the quotient $B_n/\Gamma_k(P_n)$ for all $n,k\geq3$, and proved that it is an almost-crystallographic group~\cite{GGO2}. For the case $k=3$, it was shown that the torsion of $B_n/\Gamma_3(P_n)$ is the torsion of $\sn$ that is relatively prime with $6$. 
For future reference, we summarise some of these results in the following theorem.

%\pagebreak

\begin{thm}[{\cite[Corollary~4]{GGO1}, \cite[Theorems~2 and~3]{GGO2}}]\label{th:torsionGamma23}
Let $n\geq 3$.
\begin{enumerate}
\item\label{it:torsionGamma23a} The torsion of the quotient $B_n/\Gamma_2(P_n)$ is equal to the odd torsion of $\sn$.
\item\label{it:torsionGamma23b} The group $B_n/\Gamma_3(P_n)$ has no elements of order $2$ or $3$, and if $m\in \N$ is relatively prime with $6$ then $B_n/\Gamma_3(P_n)$ possesses elements of order $m$ if and only if $\sn$ does. 
\end{enumerate}
\end{thm}

Almost nothing is known about the torsion and the finite subgroups of $B_n/\Gamma_k(P_n)$ in the case where $k>3$.
%In particular, in \cite{GGO2} the (non) existence of 
%torsion was studied. 
%\como{Independently, V.~Beck and I.~Marin also studied embeddings of finite groups in braid group 
%quotients \cite{MaV}.}

Suppose that $n\geq 3$ and $k\geq 2$. The results of~\cite{GGO1,GGO2,Ma} lead to a number of interesting problems involving the quotients $B_n/\Gamma_k(P_n)$. Given a finite group $G$, a natural question in our setting is whether it can be embedded in some $B_n/\Gamma_k(P_n)$. In order to formulate some of these problems, we introduce the following notation. Let $\ord{G}$ denote the  order of $G$, let $m(G)$ denote the least positive integer $r$ for which $G$ embeds in the symmetric group $\sn[r]$, and if $k\geq 2$, let $\ell_k(G)$ denote the least positive integer $s$, if such an integer exists, for which $G$ embeds in the group $B_{s}/\Gamma_k(P_{s})$. The integer $\ell_k(G)$ is not always defined. For example, if $G$ is of even order, \reth{torsionGamma23}(\ref{it:torsionGamma23a}) implies that $G$ does not embed in any group of the form $B_{n}/\Gamma_k(P_{n})$. However, if $\ell_k(G)$ is defined, then $m(G) \leq \ell_k(G)$ using~\reqref{sesgammak} and the fact that $P_n/\Gamma_k(P_n)$ is torsion free.

The main aim of this paper is to study the embedding of finite groups in the two quotients $B_n/\Gamma_k(P_n)$, where $k\in \brak{2,3}$. In \resec{split}, we start by recalling some results from~\cite{GGO1,GGO2} about the action of $\sn$ on certain bases of the free Abelian groups $P_{n}/\Gamma_{2}(P_{n})$ and $\Gamma_{2}(P_{n})/\Gamma_{3}(P_{n})$, which we use to obtain information about the cycle structure of elements of $\sn$ that fix elements of these bases. In \repr{null}, using cohomological arguments, we show that a short exact sequence splits if its quotient is a finite group $G$ and its kernel is a free $\Z[G]$-module. This provides a fundamental tool for embedding $G$ in our quotients. In \resec{cayley}, we prove the following result.

\begin{thm}\label{th:cayley23}
Let $G$ be a finite group, and let $k\in \brak{2,3}$. Then the group $G$ embeds in $B_{\ord{G}}/\Gamma_k(P_{\ord{G}})$ if and only if $\gcd(\ord{G},k!)=1$.
\end{thm}

The statement of \reth{cayley23} has been proved independently by Beck and Marin in the case $k=2$~\cite{MaV} using different methods within the setting of real reflection groups. This result may be viewed as a Cayley-type result for $B_n/\Gamma_k(P_n)$ since the proof makes use of the embedding of $G$ in the symmetric group $\sn[\lvert G\rvert]$, as well as \repr{embedgen} that provides sufficient conditions on the fixed points in the image of an embedding of $G$ in $\sn[m]$ for $G$ to embed in $B_n/\Gamma_k(P_n)$. If $\gcd(\ord{G},k!)=1$, it follows from this theorem that $\ell_k(G)\leq \lvert G \lvert$ by \reth{cayley23}, from which we obtain:
\begin{equation}\label{eq:ineqmG}
m(G) \leq \ell_k(G)\leq \lvert G\lvert.
\end{equation}
The analysis of the inequalities of~\reqref{ineqmG} is itself an interesting problem. Using \reth{torsionGamma23}, if $G$ is a cyclic group of prime order at least $5$ and $k$ is equal to either $2$ or $3$ then $m(G)=\ell_k(G)=\lvert G\rvert$. In~\cite[Corollary~13]{MaV}, Beck and Marin show that $m(G)=\ell_2(G)$ for any finite group of odd order in a broader setting. This result may also be obtained by applying~\cite[Corollary~7]{MaV} to~\cite[Corollary~4 and its proof]{GGO1}. We do not currently know whether there exist groups for which $m(G) < \ell_3(G)$.

In \resec{semidirect}, we study the embedding of certain finite groups in $B_n/\Gamma_k(P_n)$, where $k\in \brak{2,3}$. In the case $k=2$, our results are special cases of~\cite[Corollary~13]{MaV}, but the methods that we use are rather different from those of~\cite{MaV}, and they are also valid for the case $k=3$. In \resec{proofmain}, we consider certain semi-direct products of the form $\Z_n\rtimes_{\theta} \Z_m$ for which the action $\theta$ is injective, and we analyse their possible embedding in $B_n/\Gamma_k(P_n)$. Our main result in this direction is the following. 

%the hypothesis that the action $\map{\theta}{\Z_{m}}[\aut{\Z_{n}}]$ is injective, and their embedding in $B_n/\Gamma_k(P_n)$, where $k\in \brak{2,3}$. The main result in this direction is the following.

%\pagebreak

\begin{thm}\label{th:main}  %Let  $G=\Z_n\rtimes_{\theta} \Z_m$, where 
%with $\theta$ arbitrary. The group $G$ can not be embedded in $B_t/\Gamma_k(P_t)$ if 
%$t<n$ and arbitrary $k$. If 
%$\theta(1_m)=k$ \comj{injectivity hypothesis stated?}\como{Yes, as explained in Daciberg's message} \comd{dizem que $\theta:Z_m \to Aut(Z_n)$
%é injetora? Eu dira que sim. Pois ao dizer the temos  $Zn \rtimes_\theta	 Zm$ implica the $\theta(1_m)$ elevada a potência m da a identidade. Portanto a ordem sera um divisor de m. Depois a parte
%"(kl − 1, n) = 1, for all 1 ≤ l ≤ m − 1" diz que  a ordem não pode ser menor que m.} and $(k^l-1,n)=1$, for all $1\leq l \leq m-1$.
Let $m,n\geq 3$, let $G=\Z_n\rtimes_{\theta} \Z_m$, where $\map{\theta}{\Z_{m}}[\Z_{n}]$ is the associated action, and let $1\leq t<n$ be such that $\theta(1_m)$ is multiplication by $t$ in $\Z_{n}$. Assume that $\gcd(t^l-1,n)=1$ for all $1\leq l \leq m-1$. If $mn$ is odd (resp.\ $\gcd(mn,6)=1$) then $G$ embeds in $B_n/\Gamma_2(P_n)$ (resp.\ in $B_n/\Gamma_3(P_n)$).
%\begin{enumerate}
%\item\label{it:maina} If $mn$ is odd then $G$ embeds in $B_n/\Gamma_2(P_n)$.
%\item\label{it:mainb} If $\gcd(mn,6)=1$ then  $G$ embeds in $B_n/\Gamma_3(P_n)$.
%\end{enumerate}
\end{thm}

Using \relem{fundamental}(\ref{it:funda}), we remark that the hypotheses of \reth{main} imply that the action $\map{\theta}{\Z_m}[\aut{\Z_{n}}]$ is injective. As an application of this theorem, we obtain the following corollary.

%Note that the hypothesis of \reth{main} implies that the action $\map{\theta}{\Z_m}[\aut{\Z_{n}}]$ is injective \comj{Why? I am a bit confused since in \resec{semidirect}, we say at the beginning that we assume that $\theta$ is injective.}. 

\begin{cor}\label{cor:app}
Let $p$ be an odd prime, let $p-1=2^jd$, where $d$ is odd, let $d_{1}$ be a divisor of $d$, and let $G$ be a group of the form $\Z_{p^r}\rtimes_{\theta} \Z_{d_1}$, 
%\comj{remove the rest of the sentence?} 
where $\map{\theta}{\Z_{d_1}}[\aut{\Z_{p^r}}]$ is injective.
\begin{enumerate}[(a)]
\item\label{it:appa} If $p\geq 3$ then
 %. Then any group of the form $\Z_{p^r}\rtimes_{\theta} \Z_{d_1}$, where $\map{\theta}{\Z_{d_1}}[\aut{\Z_{p^r}}]$ is injective,
%where $\theta(1_d)$ is multiplication by some $t$ where necessarily  $\gcd(t,p)=1$, 
$G$ embeds in $B_{p^r}/\Gamma_2(P_{p^r})$.
\item\label{it:appb} If $p\geq 5$ and $d_1$ satisfies $\gcd(d_1,3)=1$ then $G$ 
%. Then any group of the form $\Z_{p^r}\rtimes_{\theta} \Z_{d_1}$, where $\map{\theta}{\Z_{d_1}}[\aut{\Z_{p^r}}]$ is injective,
%??????{\bf !!!DJ!!! ???????????????????????}
%multiplication by some $t$ where necessarily $\gcd(t,p)=1$, 
embeds in $B_{p^r}/\Gamma_3(P_{p^r})$.
\end{enumerate}
\end{cor}

Since the group $\Z_{p^{r}}$ cannot be embedded in $\sn[m]$ for any $m<p^{r}$, the groups of \reco{app} satisfy $m(G)=\ell_{k}(G)=p^{r}$, where $k\in \brak{2,3}$, so the results of this corollary are sharp in this sense, and are coherent with those of~\cite[Corollary~13]{MaV} in the case $k=2$. Further, the groups that appear in~\cite[Corollary~3.11]{Ma} correspond to the case of where $r=1$, $p\equiv 3\bmod{4}$, and $d_1=(p-1)/2$ is odd. Hence \reco{app} generalises Marin's result to the case where $p$ is any odd prime and $d_{1}$ is the greatest odd divisor of $p-1$, and more generally, in the case $k=2$, the family of groups obtained in \reth{main} extends even further that of the Frobenius groups of~\cite[Corollary 3.11]{Ma}.

At the end of the paper, in \resec{further}, we give explicit embeddings of the two non-Abelian groups of order $27$ in $B_{9}/\Gamma_{2}(P_{9})$. Neither of these groups satisfies the hypotheses of \reth{main}.
% \comj{delete the rest of this sentence, because it is a consequence of~\cite{MaV}. But perhaps we can repeat that our methods are different? We could also mention though that are %embeddings are explicit. !!!DO!!!! Here is a proposal for a new version of this sentence "and their embedding ....for other groups." NEW  SENTENCE:  
The fact that they embed in $B_{9}/\Gamma_{2}(P_{9})$  follows from the more general result of~\cite[Corollary~13]{MaV}, but our approach is different to that of~\cite{MaV}. Within our framework, it is natural to study these two groups, first because with the exception of the Frobenius group of order $21$ analysed in~\cite{GGO1}, they are the smallest non-Abelian groups of odd order, and secondly because they are of order $27$, so are related to the discussion in \resec{proofmain} on groups whose order is a prime power. The direct embedding of these groups in $B_{9}/\Gamma_{2}(P_{9})$ is computationally difficult due to the fact that the kernel $P_{9}/\Gamma_{2}(P_{9})$ of~\reqref{sesgammak} is of rank $36$, but we get round this problem by first considering an embedding in a quotient where the corresponding kernel is free Abelian of rank $9$, and then by applying \repr{null}. We believe that this technique will prove to be useful for other groups.

  If $n\geq 3$, it follows from~\cite[Corollary~13]{MaV} that the isomorphism classes of the  finite subgroups of $B_n/\Gamma_2(P_n)$ are in bijection with those of the subgroups of $\sn$ of odd order. The study of the finite non-cyclic subgroups of $B_n/\Gamma_3(P_n)$ constitutes work in progress.

\subsection*{Acknowledgements}

%\comj{This added. Oscar, verify your part in this acknowledgements} 
This work took place during several visits to the Departamento de Matem\'atica, Universidade Federal de Bahia, the Departamento de Matem\'atica do IME~--~Universidade de S\~ao Paulo and to the Laboratoire de Math\'ematiques Nicolas Oresme UMR CNRS~6139, Universit\'e de Caen Normandie. The first  and the third authors were  partially supported by  the FAPESP Projeto Tem\'atico ``Topologia Alg\'ebrica, Geom\'etrica e Diferencial'' 2016/24707-4  (Brazil), 
and the first two authors were partially supported by the CNRS/FAPESP PRC project n\textsuperscript{o}~275209 (France) and n\textsuperscript{o}~2016/50354-1 (Brazil). The third author was initially supported by 
%for this paper occurred\ldots.  , and also by 
a project grant n\textsuperscript{o}~151161/2013-5 from CNPq. 
%and then partially supported by the FAPESP Projeto Tem\'atico ``Topologia Alg\'ebrica, Geom\'etrica e Diferencial'' n\textsuperscript{o}~2016/24707-4  (Brazil). 
We would also like to thank Vincent Beck and Ivan Marin for sharing their results with us. 

\section{Preliminaries}\label{sec:split} 

%\comj{Check the references to the final version of~\cite{GGO2}.} {\bf !!!DO!!! Was done. Also the title of 
%the reference GGO2 was updated} \comj{The following has been rewritten somewhat, and developed in some places.}{\bf  !!!DO!!! Agree} 

In this section, we recall several results concerning the torsion of the groups
%{\bf !!!DO!!! typo replace 2 by k} 
$B_n/\Gamma_k(P_n)$, where $k\in \brak{2,3}$, as well as some group-cohomological facts from~\cite{Br} that will be used in this paper. We first state the following result from~\cite{GGO2} that we will require. %In this section we will show two  results similar to  the classical Cayley Theorem. The two cases are quite similar. 
%In the case of the groups  $B_n/\Gamma_2(P_n)$  we will show \reth{cayley2}. 
%% This result  is the best possible in the sense that if a finite group has even torsion, we know from 
%%Section \ref{split} that it cannot be embedded in  $B_n/\Gamma_2(P_n)$, for any natural number $n$.
%In  the case of  the groups $B_n/\Gamma_3(P_n)$ we
%  % again let $G$ be a finite group. Denote by $\ord{G}$ the cardinality of the group.   
%  will show \reth{cayley3}.
%This result  is the best possible in the sense that if for a  finite group $G$ we have   $(\ord{G},6)\ne 1$ means that it %has either an element of order 2 or order 3. If it has an element of torsion order  2 implies that also 
% $B_n/\Gamma_2(P_n)$ has an element of torsion 2, which is not possible. Also from Section \ref{split}  
 % $B_n/\Gamma_3(P_n)$ does not admit an element of torsion order $3$.  So under the hypothesis
 %  $(\ord{G},6)\ne 1$, $G$  cannot be embedded in  $B_n/\Gamma_3(P_n)$, for any natural number $n$.
%{\bf !!!DO!!! There is no need to proof Proposition 8. In fact it suffices to write down Lemma 11  from GGO2. See the new version below}  We need the following elementary result from \cite {GGO2}.
 
\begin{prop}~\cite[Lemma~11]{GGO2} \label{prop:tf}
\begin{enumerate}
\item\label{it:ptorsiona} Let
% \comj{This item added. Check in the rest of the text if we should to refer to it.!!!DO!!! We cheked and see page 2} 
$n,k\geq 2$. Then the group $P_n/{\Gamma_k(P_n)}$ is torsion free.
\item\label{it:ptorsionb} Let $n\geq 3$, let $k\geq l\geq 1$, and let $G$ be a finite group. If $B_n/\Gamma_k(P_n)$ possesses a (normal) subgroup isomorphic to $G$ then $B_n/\Gamma_l(P_n)$ possesses a (normal) subgroup isomorphic to $G$. In particular, if $p$ is prime, and if $B_n/\Gamma_l(P_n)$ has no $p$-torsion then $B_n/\Gamma_k(P_n)$ has no $p$-torsion.
\end{enumerate}
%Let $n,k\geq 2$.
%\begin{enumerate}
%\como{cases $n,k=1$ are trivially true}.
%\item\label{it:tfa} The group $P_n/{\Gamma_k(P_n)}$ is torsion free. \comj{Refer to the lemma in~\cite{GGO2}.}
%\item Let $G$ be a finite group. If $G$ embeds in $B_n/\Gamma_k(P_n)$, then it embeds in the symmetric group 
%$\sn$. \comj{This looks like~\cite[Lemma~11]{GGO2} with $l=1$.}
%\end{enumerate}
%\qed 
\end{prop} 

Note that the first part of \repr{tf} follows from papers by Falk and Randell~\cite[Theorem~4.2]{FRinv} and Kohno~\cite[Theorem~4.5]{Ko} who proved independently that for all $n\geq 2$ and $k\geq 1$,  the group $\Gamma_k(P_n)/\Gamma_{k+1}(P_n)$ is free Abelian of finite rank, the rank being related to the Poincar\'e polynomial of certain hyperplane complements.

% In  \cite{GGO1} we characterized the torsion elements of the groups $B_n/\Gamma_2(P_n)$. They  have no 2-torsion elements and have $k$ torsion elements as long $k$ is odd and the symmetric group $\sn$ has  a $k$ torsion element. More precisely we show.\\
%{\bf  Theorem}\\
%
% From \cite{GGO2} we show that the groups $B_n/\Gamma_3(P_n)$ have no 2-torsion elements as well 3-torsion elements. But they have  $k$ torsion elements as long $(k,6)=1$ and the symmetric group $\sn$ has 
% $k$ torsion elements. More precisely we show.
% 
%{\bf  Theorem 2} {\it Let $n,k\geq 3$. Then the quotient group $B_n/\Gamma_k(P_n)$ has no elements of order $2$ nor of order $3$.}
%
%and 
%
%
%{\bf  Theorem 3} {\it Let $n\geq 5$ and let $\tau=p_1^{r_1}\cdots p_t^{r_t}$ be a positive integer, 
%with $p_i$ different primes, such that $(\tau,6)=1$. 
%The group $B_n/\Gamma_3(P_n)$ admit elements of torsion $\tau$ (with cycle type 
%$(p_1^{r_1},\ldots,p_t^{r_t})$) if and only if $\sn$ has torsion $\tau$.} 

It is well known that a set of generators for $P_{n}$ is given by the set $\brak{A_{i,j}}_{1\leq i<j\leq n}$~\cite{Ha}. If $j>i$ then we take $A_{j,i}=A_{i,j}$. By abuse of notation, for $k\geq 2$ and $1\leq i<j\leq n$, we also denote the image of $A_{i,j}$ under the canonical projection $P_{n}\to P_{n}/\Gamma_{k}(P_{n})$ by $A_{i,j}$. The groups $P_n/\Gamma_2(P_n)$ and $\Gamma_2(P_n)/\Gamma_3(P_n)$ are free Abelian groups of finite rank $n(n-1)/2$ and $n(n-1)(n-2)/6$ respectively~\cite[Theorem~4.2]{FRinv}. By~\cite[Section 3, p.~399]{GGO1} (resp.\ \cite[equation~(17)]{GGO2}), a basis for $P_n/\Gamma_2(P_n)$ (resp.\ $\Gamma_2(P_n)/\Gamma_3(P_n)$) is given by:
\begin{equation}\label{eq:defBBprime}
\text{$\mathcal{B}=\setl{A_{i,j}}{1\leq i<j\leq n}$ (resp.\ by $\mathcal{B}'=\setl{\alpha_{i,j,k}}{1\leq i<j<k\leq n}$),}
\end{equation}
where $\alpha_{i,j,k}=[A_{i,j},A_{j,k}]$.
%{\bf !!!DO!!! I added "Section 3" where this information was given in the section 3 of this manuscript. Then I removed from there the definition of $\mathcal{B}$}
%The symmetric group $\sn$ acts on each of these groups as follows. 
If $\tau\in \sn$, $A_{i,j}\in \mathcal{B}$ and $\alpha_{i,j,k}\in \mathcal{B}'$ then by~\cite[Proposition~12]{GGO1} and~\cite[equation~(8)]{GGO2}, we have:
\begin{equation}\label{eq:actionSn}
\text{$\tau\cdot A_{i,j}=A_{\tau^{-1}(i), \tau^{-1}(j)}$ and $\tau\cdot \alpha_{i,j,k}= \tau \cdot [A_{i,j},A_{j,k}]= [A_{\tau^{-1}(i), \tau^{-1}(j)},A_{\tau^{-1}(j), \tau^{-1}(k)}]$.}
\end{equation}
%\comj{This modified slightly:} 
The following lemma implies that $\sn$ acts on $\mathcal{B}$ and $\widehat{\mathcal{B}}'$ respectively, where $\widehat{\mathcal{B}}'=\mathcal{B}' \cup \mathcal{B}'^{-1}$. In each case, the nature of the action gives rise by linearity to an action of $\sn$ on the whole group. 
% Each of these actions on the basis elements then extends linearly to an action on the whole group. 
%Let $\widehat{\mathcal{B}}'=\mathcal{B}' \cup \mathcal{B}'^{-1}$. 
%In the following result, we see that the 
%% actions given by~\reqref{actionSn}
%\comdo{Suggestion of change:} extended actions (obtained from ~\reqref{actionSn}) restrict to actions of $\sn$ on $\mathcal{B}$ and $\widehat{\mathcal{B}}'$ respectively, and 
We also obtain some information about the stabilisers of the elements of $\mathcal{B}$ and $\widehat{\mathcal{B}}'$. This will play a crucial rôle in the proof of  Proposition \ref{prop:embedgen}.

\begin{lem}\label{lem:freeacjoin}%\label{lem:freeac0}%\label{lem:freeac} 
Let $n\geq 2$, and let $\tau\in \sn$.
\begin{enumerate}
\item\label{it:freeacjoina} Let $A_{i,j}$ be an element of the basis $\mathcal{B}$ of $P_n/\Gamma_2(P_n)$, where $1\leq i<j\leq n$. Then the element $\tau \cdot A_{i,j}$ given by the action of $\sn$ on $P_n/\Gamma_2(P_n)$ belongs to $\mathcal{B}$. Further, if $\tau\cdot A_{i,j}=A_{i,j}$ then the cycle decomposition of $\tau$ either contains a transposition, or at least two fixed elements. 
 
\item\label{it:freeacjoinb} Let $\alpha_{i,j,k}$ be an element of the basis $\mathcal{B}'$ of $\Gamma_2(P_n)/\Gamma_3(P_n)$, where $1\leq i<j<k\leq n$. Then the element $\tau\cdot \alpha_{i,j,k}$ given by the action of $\sn$ on $\Gamma_2(P_n)/\Gamma_3(P_n)$ belongs to $\widehat{\mathcal{B}}'$. Further, if $\tau\cdot \alpha_{i,j,k}\in \bigl\{\alpha_{i,j,k},\alpha_{i,j,k}^{-1}\bigr\}$ then the cycle decomposition of $\tau$ contains either a transposition, or a $3$-cycle, or at least three fixed elements.
\end{enumerate}
\end{lem}

%\pagebreak

\begin{proof}\mbox{}
\begin{enumerate}
\item The first part follows from~\reqref{actionSn}. 
%For the second part, 
If $1\leq i<j\leq n$ and $\tau\in \sn$ are such that
%By~\cite[Proposition~12]{GGO1}, the first part follows easily, and 
%The first part is clear from~\cite[Proposition~12]{GGO1}. Further, 
%if 
$\tau\cdot A_{i,j}=A_{i,j}$ then $\tau(\brak{i,j})=\brak{\tau(i),\tau(j)}=\brak{i,j}$,
% also by~\cite[Proposition~12]{GGO1}, 
which implies the second part of the statement.
\item The first part is a consequence of~\cite[equation~(16)]{GGO2}. Now suppose that $\tau\cdot \alpha_{i,j,k} \in \bigl\{\alpha_{i,j,k},\alpha_{i,j,k}^{-1}\bigr\}$, where $1\leq i<j<k\leq n$ and $\tau\in \sn$. By~\cite[equation~(18)]{GGO2}, we have $\tau(\brak{i,j,k})=\brak{\tau(i),\tau(j),\tau(k)}=\brak{i,j,k}$, from which we deduce the second part.\qedhere
\end{enumerate}
\end{proof}

% The action of $\sn$ on $P_n/\Gamma_2(P_n)$, $\psi\colon\sn\times P_n/\Gamma_2(P_n)\to P_n/\Gamma_2(P_n)$ where $\psi(\tau,\omega)=\tau(\omega)$, is induced from the action given in \cite[Proposition~12]{GGO1},
% and restricts to an action on $\mathcal{B}$ that we will use in \relem{freeac0}.  
% The action of $\sn$ on $\Gamma_2(P_n)/\Gamma_3(P_n)$, $\eta\colon\sn\times \Gamma_2(P_n)/\Gamma_3(P_n)\to \Gamma_2(P_n)/\Gamma_3(P_n)$ where $\eta(\tau,\omega)=\tau(\omega)$, is induced from the action given in equations (8), (16) and (18) of \cite{GGO2}, and restricts to an action on $\mathcal{B'}$, where $\mathcal B'=\{\alpha^{\pm1} | \alpha\in \mathcal{B}\}$, that will be used .

%, which may be seen to be an element of $\widehat{\mathcal{B}}'$ using~\cite[equation~(16)]{GGO2}. 

\relem{freeacjoin} implies that if $G=\sn$ then $\mathcal{B}$ and $\widehat{\mathcal{B}}'$ are $G$-sets, and the action of $G$ on each of these sets extends to a $\Z$-linear action of $G$ on the free $\Z$-modules $\Z \mathcal{B}$ and $\Z {\mathcal{B}}'$ respectively, with respect to the embedding of $\widehat{\mathcal{B}}'$ in $\Z{\mathcal{B}}'$ given by $\alpha_{i,j,k}\mapsto \alpha_{i,j,k}$ and $\alpha_{i,j,k}^{-1} \mapsto (-1)\ldotp\alpha_{i,j,k}$, 
%\como{the correct is $\Z{\mathcal{B}}'$ instead $\Z\widehat{\mathcal{B}}'$} \comj{But:} \comd{Talvez se precise dizer que iremos identificar o INVERSO de $\alpha_{i,j,k}$ com o elemento de anel de group $(-1)\alpha_{i,j,k}$.} 
the underlying free Abelian groups
% \comj{or `the underlying free Abelian groups $\Z \mathcal{B}$ and $\Z {\mathcal{B}}'$'?}
 being naturally identified with $P_n/\Gamma_2(P_n)$ and $\Gamma_2(P_n)/\Gamma_3(P_n)$ respectively. We conclude that $P_n/\Gamma_2(P_n)$ and $\Gamma_2(P_n)/\Gamma_3(P_n)$ each admit a $G$-module structure inherited by the action of $\sn$ on $\mathcal{B}$ and $\widehat{\mathcal{B}}'$ respectively.

Given a group $G$, let $\Z[G]$ denote its group ring. The underlying Abelian group, also denoted by $\Z[G]$, may be regarded as a $\Z[G]$-module (or as a $G$-module) via the multiplication in the ring $\Z[G]$ (see~\cite[Chapter~I, Sections~2~and~3]{Br} for more details), 
%We regard $\Z[G]$ as a $G$-module via the action given by left multiplication by $G$ \comj{Is there some repetition %here with the previous sentence?}, 
namely:
%\comj{which action exactly?}{\bf !!!DO!!! We added the action} 
%determined by the elements of $G$ where 
%\como{I added $i=$\, in the subindex:} 
\begin{equation}\label{eq:actZG}
\text{$g\cdot \sum_{i=1}^{m}\, n_ig_{i} =\sum_{i=1}^{m}\, n_i(gg_i)$ for all $m\in \N$, $g,g_{1},\ldots,g_{m}\in G$ and $n_{1},\ldots,n_{m}\in \Z$.}
\end{equation}
%\como{I added $\in$:} 
The cohomology of the group $G$ with coefficients in $\Z[G]$ regarded as a $G$-module is well understood. In the case that $G$ is finite, we have the following result concerning its embedding in certain extensions whose kernel is a free $\Z[G]$-module. First recall that if $M$ is an Abelian group that fits into an extension of the following form:  
\begin{equation}\label{eq:seqsplitM}
1\to M \to E\to G\to 1,
\end{equation}
then $M$ is also a $\Z[G]$-module, the action being given by~\reqref{seqsplitM}.

\begin{prop}\label{prop:null}
Let $G$ be a finite group. Given an extension of the form~\reqref{seqsplitM}, suppose that $M$ is a free $\Z[G]$-module. Then the short exact sequence~\reqref{seqsplitM} splits. In particular, $G$ embeds in $E$ as a subgroup, and the restriction of the projection $E \to G$ to the embedded copy of $G$ is an isomorphism.
\end{prop}

%\comj{Use the notation $J$? Is $G$ finite here?}
%\begin{equation*}
%1\to \bigoplus \Z[G] \to E\to G\to 1,
%\end{equation*}
%where the action of $G$ on the Abelian group $\bigoplus \Z[G]$ is as above \comj{`as above' means?}, $G$ embeds in $E$ as a subgroup, and the restriction of the projection $E \to G$ to the embedded group is an isomorphism. \comj{Is this just saying that the short exact sequence splits, or is there something else?}
%\end{prop}

\begin{proof}
First suppose that $M\cong \Z[G]$. By~\cite[equation~6.5, p.~73]{Br},
the group $H^{\ast}(G,\Z[G])$ is trivial for all $\ast\geq 1$. In particular, $H^2(G,\Z[G])=0$, which implies that any extension of the form~\reqref{seqsplitM} with $M=\Z[G]$
%\begin{equation}\label{eq:seqsplit0}
%1\to \Z[G] \to E\to G\to 1
%\end{equation}
is split~\cite[Chapter~IV, Theorem~3.12]{Br}, where the action of the quotient on the kernel turns $\Z[G]$ into a 
%\como{Daciberg, will be interesting point out which part was modified. I suggest write: Begin new version. (and below close with End new version)}{\bf !!!D!!! new version} 
$\Z[G]$-module that is isomorphic to the one-dimensional free $\Z[G]$-module.
%, is a split extension.
% (see~\cite[Chapter~IV, Theorem~3.12]{Br}).   
Now suppose that $M$ is an arbitrary free $\Z[G]$-module whose $\Z[G]$-module structure is defined by~\reqref{seqsplitM}. 
%Further, if $M$ 
% of the extension~\reqref{seqsplit0} 
%is a free $\Z[G]$-module, 
So $M\cong \oplus _{J}\,\Z[G]$ 
%\comj{What is $M$?}{\bf !!!D!!! See the sentence completely revised until before Propostion 7} 
as a $\Z[G]$-module for some set $J$, and 
%$H^2(G, \bigoplus_{J} \Z[G])=0$ because 
$H^2(G, M)\cong H^2(G, \bigoplus_{J}\, \Z[G])\cong \bigoplus_J\, H^2(G,  \Z[G])=0$ by the first part of the proof. 
%Therefore, given an extension of the form:
%% also  holds   that for an extension of the form 
%\begin{equation}\label{eq:seqsplit}
%1\to A \to E\to G\to 1,
%\end{equation}
%where $A$ is isomorphic to $M$ as a $\Z[G]$-module, $A$ being equipped with the $\Z[G]$-module structure defined by~\reqref{seqsplit},
%the action of $G$ on the Abelian group $M$ given by 
%the short exact sequence,   it 
%is isomorphic to $M$ as  $\Z[G]$-module, 
The short exact sequence~\reqref{seqsplitM} splits as in the case $M\cong \Z[G]$.
%\comj{or `where equipped with the $\Z[G]$-module structure defined by  the short exact sequence, $M$ may be regarded as a free $\Z[G]$-module'?}{\bf !!!D!!! See the new version}
%, so is isomorphic to $\bigoplus _{J}\Z[G]$ as $\Z[G]$-module for some set $J$. 
\end{proof}

\section{Cayley-type results for subgroups of $B_n/\Gamma_k(P_n)$, $k=2,3$}\label{sec:cayley}

Let $k\in \brak{2,3}$. In this section, we prove \reth{cayley23}
%Theorems~\ref{th:cayley2} and~\ref{th:cayley3} 
that may be viewed as an analogue of Cayley's theorem for $B_n/\Gamma_k(P_n)$. The following proposition will be crucial in the proofs of  Theorems~\ref{th:cayley23} and~\ref{th:main}.

%two results for $B_n/\Gamma_k(P_n)$ that are similar in nature to Cayley's theorem for $\sn$.

%\begin{proof}
%Let $n,k\geq 2$.
%\begin{enumerate}
%\item The proof is by induction on $k$. If $k=2$ then $P_n/\Gamma_2(P_n)\cong \Z^{n(n-1)/2}$, which implies the %result in this case. Suppose then that 
%the result holds for some $k\geq 2$. Since the quotient group $\Gamma_k(P_n)/\Gamma_{k+1}(P_n)$ is torsion
% free %\comj{reference?}, it follows from the following short exact sequence:
%\begin{equation}\label{eq:torsionfree}
%1 \to \frac{\Gamma_k(P_n)}{\Gamma_{k+1}(P_n)} \to  \frac{P_n}{\Gamma_{k+1}(P_n)} {\to}  
%\frac{P_n}%{\Gamma_k(P_n)} \to 1,
%\end{equation}
% \cite[Equation~23]{GGO2} 
%that the group $P_n/\Gamma_{k+1}(P_n)$ is torsion free.
%\item The result follows from part~(\ref{it:tfa}) and the short exact sequence~\reqref{sesgammak}.\qedhere
%\end{enumerate}
 %\end{proof}
 %  nds are the same, but for the best of our knowledge it is not knowing 
%if the statement holds or not. 
%\comj{or `We do not currently know whether this inequality can be an equality or not.'?}{\bf !!!DO!!! Ok}

%\comj{The following is new, and gives a more general version of the result that we require to prove (the previous) Theorems~1 and~2, as well as \reth{main}. See if you agree with the proof. I have commented out the previous version.}

\begin{prop}\label{prop:embedgen}
Let $k\in \brak{2,3}$, let $G$ be a finite group whose order is relatively prime with $k!$, let $m\geq 3$, and let $\map{\phi}{G}[{\sn[m]}]$ be an embedding. Assume that for all $g\in G\setminus \brak{e}$, the cycle decomposition of $\phi(g)$ contains at most $k-1$ fixed elements. Then the group $G$ embeds in $B_{m}/\Gamma_{k}(P_{m})$. 
\end{prop}

\begin{proof}
Assume first that $k=2$, so $\ord{G}$ is odd. Let $\widetilde{G}$ be the (isomorphic) image of $G$ by $\phi$ in $\sn[m]$. Taking the inverse image by $\overline{\sigma}$ of $\widetilde{G}$ in~\reqref{sesgammak} with $n=m$ and $k=2$ gives rise to the following short exact sequence: 
\begin{equation}\label{eq:case1a}
1\to P_{m}/\Gamma_2(P_{m})\to \overline{\sigma}^{-1}(\widetilde{G}) \xrightarrow{\overline{\sigma}\left\lvert_{\overline{\sigma}^{-1}(\widetilde{G})}\right.} \widetilde{G}\to 1.
\end{equation}
From \resec{split}, $G$, and hence $\widetilde{G}$, acts on the free Abelian group $P_{m}/\Gamma_2(P_{m})$ of rank $m(m-1)/2$, and the restriction of this action to the basis $\mathcal{B}$ is given by~\reqref{actionSn}. Let $1\leq i<j\leq m$, and let $g\in G$ be such that $\phi(g) \cdot A_{i,j}=A_{i,j}$. Since $\ord{G}$ is odd, the cycle decomposition of $\phi(g)$ contains no transposition, and by \relem{freeacjoin}(\ref{it:freeacjoina}) and the hypothesis on the fixed points of $\phi(g)$, we see that $g=e$. So for all $1\leq i<j\leq m$, the orbit of $A_{i,j}$ contains exactly $\ord{G}$ elements. In particular, $\ord{G}$ divides $m(m-1)/2$, and $\mathcal{B}$ may be decomposed as the disjoint union of the form $\coprod_{k=1}^{m(m-1)/2\ord{G}}\, \mathcal{O}_k$, where each $\mathcal{O}_k$ is an orbit of length $\ord{G}$. For $k=1,\ldots, m(m-1)/2\!\ord{G}$, let $e_{k}\in \mathcal{O}_k$, and let $H_{k}$ denote the subgroup of $P_{m}/\Gamma_2(P_{m})$ generated by $\mathcal{O}_k$. Then $P_{m}/\Gamma_2(P_{m})\cong \bigoplus_{k=1}^{m(m-1)/2\ord{G}}\, H_{k}$, and for all $x\in \mathcal{O}_k$, there exists a unique element $g\in G$ such that $\phi(g) \cdot e_{k}=x$. Thus the map that to $x$ associates $g$ defines a bijection between $\mathcal{O}_k$ and $G$. Since $\mathcal{O}_k$ is a basis of $H_{k}$, if $h\in H_{k}$, there exists a unique family of integers $\brak{q_{g}}_{g\in G}$ such that $h=\prod_{g\in G} (\phi(g)\cdot e_{k})^{q_{g}}$, and the map $\map{\Phi}{H_{k}}[\Z[G]]$ defined by $\Phi(h)= \sum_{g\in G} q_{g} g$ may be seen to be an isomorphism. Further, via $\Phi$, the action of $G$ on $H_{k}$ corresponds to the usual action of $G$ on $\Z[G]$. More precisely, if $\gamma\in G$, then:
\begin{align*}
\Phi(\phi(\gamma) \cdot h)&= \Phi\Bigl(\phi(\gamma) \cdot \prod_{g\in G} (\phi(g) \cdot e_{k})^{q_{g}} \Bigr)=\Phi\Bigl(\prod_{g\in G} (\phi(\gamma g)\cdot e_{k})^{q_{g}}\Bigr)\\
&=\sum_{g\in G} q_{g} \gamma g=\gamma \cdot \Phi(h),
\end{align*}
the action of $\gamma$ on $\Phi(h)$ being given by~\reqref{actZG}. Hence $P_{m}/\Gamma_2(P_{m})\cong \bigoplus_{1}^{m(m-1)/2\ord{G}}\, \Z[G]$ as $\Z[G]$-modules, and by \repr{null}, we conclude that the extension~\reqref{case1a} splits. Thus $G$ is isomorphic to a subgroup $\widehat{G}$ of $\overline \sigma^{-1}(\widetilde{G})$, which in turn is a subgroup of $B_{m}/\Gamma_2(P_{m})$, and this proves the result in the case $k=2$. Now suppose that $k=3$. Since $\gcd(\ord{G},6)=1$, $\ord{G}$ is odd, and as above, $G$ is isomorphic to the subgroup $\widehat{G}$ of $B_{m}/\Gamma_2(P_{m})$. Consider the following extension:
\begin{equation*}
1\to \Gamma_2(P_{m})/\Gamma_3(P_{m})\to B_{m}/\Gamma_3(P_{m})\stackrel{\rho}{\to} B_{m}/\Gamma_2(P_{m})\to 1,
\end{equation*}
where $\map{\rho}{B_{m}/\Gamma_3(P_{m})}[B_{m}/\Gamma_2(P_{m})]$ denotes the canonical projection. Taking the inverse image of $\widehat{G}$ by $\rho$ gives rise to the following short exact sequence:
\begin{equation}\label{eq:gamma23Ga}
1\to \Gamma_2(P_{m})/\Gamma_3(P_{m})\to \rho^{-1}(\widehat{G})\xrightarrow{\rho\left\lvert_{\rho^{-1}(\widehat{G})}\right.} \widehat{G}\to 1.
\end{equation}
%\begin{equation}\label{eq:case1}
%1\to \Gamma_2(P_{\ord{G}})/\Gamma_3(P_{\ord{G}})\to \sigma^{-1}(G)\to G\to 1
%\end{equation}
%obtained from the sequence 
%\begin{equation}
%1\to \frac{P_{\ord{G}}}{\Gamma_2(P_c)}\to B_c\stackrel{\sigma}{\to} \sn[c]\to 1.
%\end{equation}
%\begin{equation}
%1\to \Gamma_2(P_{\ord{G}})/\Gamma_3(P_{\ord{G}})\to B_{\ord{G}}\stackrel{\sigma}{\to} \mathcal{S}_{\ord{G}}\to 1.
%\end{equation}
Let $\map{\phi'}{\widehat{G}}[{\sn[m]}]$ denote the embedding of $\widehat{G}$ in $\sn[m]$ given by composing $\phi$ by an isomorphism between $\widehat{G}$ and $G$. Then $\widehat{G}$  acts on the kernel $\Gamma_2(P_{m})/\Gamma_3(P_{m})$ of~\reqref{gamma23Ga} via~\reqref{actionSn}. Since $\gcd(\bigl\lvert\widetilde{G}\bigr\rvert,6)=1$, for all $\widehat{g}\in \widehat{G} \setminus \brak{e}$, the cycle decomposition of $\phi'(\widehat{g})$ contains neither a transposition nor a $3$-cycle, and by hypothesis, $\phi'(\widehat{g})$ contains at most $2$ fixed elements. It follows from \relem{freeacjoin}(\ref{it:freeacjoinb}) that if $\phi'(\widehat{g})\cdot \alpha_{i,j,k}\in \bigl\{\alpha_{i,j,k},\alpha_{i,j,k}^{-1}\bigr\}$, where $1\leq i<j<k\leq n$ and $\widehat{g}\in \widehat{G}$, then $\widehat{g}=e$. In particular, the orbits of $\alpha_{i,j,k}$ and $\alpha_{i,j,k}^{-1}$ are disjoint, every orbit contains exactly $\ord{G}$ elements, and thus $\ord{G}$ divides $m(m-1)(m-2)/6$. So there exists a basis of $\Gamma_2(P_{m})/\Gamma_3(P_{m})$ that is the disjoint union of $m(m-1)(m-2)/6\!\ord{G}$ orbits of elements of $\widehat{B}'$, and that for all $1\leq i<j<k\leq n$, contains exactly one element of $\bigl\{\alpha_{i,j,k},\alpha_{i,j,k}^{-1}\bigr\}$. As in the case $k=2$, we conclude that the short exact sequence~\reqref{gamma23Ga} splits, and that $G$ embeds in $B_{m}/\Gamma_3(P_{m})$.
\end{proof}

% The following obvious extension of   Theorem 
%\reth{main} may be extended in a straightforward manner as follows. \comj{How is this related to \reth{main} and \relem{fundamental}?}
%the result above \como{Maybe cite it?}
%holds: 

\begin{rem}
An efficient way to use \repr{embedgen} is as follows. Let $\map{\phi}{G}[{\sn[n]}]$ be an embedding, and for an order-preserving inclusion $\iota\colon\thinspace \brak{1,2,\ldots,m} \lhra \brak{1,2,\ldots,n}$, where  $m<n$, consider the embedding $\sn[m] \to \sn$. Suppose that the  homomorphism $\phi$ factors through $\sn[m]$, and let $\map{\phi'}{G}[{\sn[m]}]$ be the factorisation. It may happen that the hypotheses of \repr{embedgen} hold for $\phi'$ but not for $\phi$. In this case, we may apply this proposition to $\phi'$ to conclude the existence of an embedding of $G$ in $B_m/\Gamma_k(P_m)$, which in turn implies that $G$ embeds in $B_n/\Gamma_k(P_n)$.
\end{rem}
%consider the suppose that  
% a subgroup $G$ of $\sn$  is given 
%such that $\phi(G)$ is  a subgroup $G'$ of $\sn[m]$,
%permutation group \comj{subgroup?}\como{Yes} $G$ of $\sn$  is given such that  $G$ is the image of a permutation group \comj{subgroup?}\como{Yes} $G'$ of $\sn[m]$,  
%where $m<n$, for some embedding $\sn[m] \to \sn$ given by % $j\to i_j$ order preserved. 
%If the hypothesis of \relem{fundamental} holds for $G'$ then there is an embedding of $G$ in $B_n/\Gamma_2(P_n)$.
%\end{cor}

%\comj{To prove (the previous) Theorems~1 and~2, we now just need to describe the Cayley action, and invoke \repr{embedgen}. So I have combined these two theorems into one theorem.}

We are now able to prove \reth{cayley23}.

\begin{proof}[Proof of \reth{cayley23}]
Let $G$ be a finite group, and let $k\in \brak{2,3}$. If $G$ embeds in $B_{\ord{G}}/\Gamma_k(P_{\ord{G}})$ then \reth{torsionGamma23} implies that $\gcd(\ord{G},k!)=1$. Conversely, suppose that $\gcd(\ord{G},k!)=1$. Consider the classical embedding of $G$ in $\sn[\ord{G}]$ that is used in the proof of Cayley's theorem (note that we identify $\sn[G]$ with $\sn[\ord{G}]$). More precisely, let $\map{\psi}{G\times G}[G]$ denote the action of $G$ on itself given by left multiplication. For all $g\in G$, the map $\map{\psi_{g}}{G}$ defined by $\psi_{g}(h)=\psi(g,h)=gh$ is a permutation of $G$, and the map $\map{\Psi}{G}[{\sn[\ord{G}]}]$ defined by $\Psi(g)=\psi_{g}$ is an injective homomorphism, so $\widetilde{G}=\setl{\psi_{g}}{g\in G}$ is a subgroup of $\sn[\ord{G}]$ that is isomorphic to $G$. The action $\psi$ is free: if $h\in G$ then $h=\psi_{g}(h)$ if and only if $g=e$. In particular, if $g\neq e$ then $\psi_{g}$ is fixed-point free, and so the permutation $\Psi(g)$ is fixed-point free for all $g\in G \setminus \brak{e}$. Taking $m=\ord{G}$, the hypotheses of \repr{embedgen} are satisfied for the embedding $\map{\Psi}{G}[{\sn[\ord{G}]}]$, and we conclude that $G$ embeds in $B_{\ord{G}}/\Gamma_k(P_{\ord{G}})$.
\end{proof}

\section{Embeddings of some semi-direct products in $B_n/\Gamma_k(P_n)$, $k\in \brak{2,3}$}\label{sec:semidirect}

Let $m,n\in \N$, and let $k\in \brak{2,3}$. 
%\comj{Slight change here:} 
In this section, we study the problem of embedding groups of the form $\Z_n\rtimes_{\theta} \Z_m$ in $B_n/\Gamma_k(P_n)$, where the representation $\map{\theta}{\Z_m}[\aut{\Z_n}]$ is taken to be injective. With additional conditions on $\theta$, in \resec{proofmain}, we prove \reth{main}. In \resec{further}, we study the two non-Abelian groups of order $27$. The first such group is of the form $\Z_n\rtimes_{\theta} \Z_m$, where $\theta$ is injective, but the additional conditions of \reth{main} are not satisfied. The second such group is not of the form $\Z_n\rtimes_{\theta} \Z_m$. We prove that both of these groups embed in $B_{9}/\Gamma_{2}(P_{9})$. In the case of the first group, this shows that the hypotheses of \reth{main} are sufficient to embed $\Z_n\rtimes_{\theta} \Z_m$ in $B_n/\Gamma_k(P_n)$, but not necessary. With respect to~\reqref{ineqmG}, these groups also satisfy $m(G)=\ell_k(G) <\lvert G\rvert$, which is coherent with~\cite[Corollary~13]{MaV} in the case $k=2$. 

%The first one has the format  $\map{\theta}{\Z_m}[\aut{\Z_n}]$,  but   
% %\comj{One of these isn't of the form $\Z_n\rtimes_{\theta} \Z_m$, so perhaps we should modify this paragraph. %Actually, here might be a good point to put the `out of place' paragraph that appears in \resec{cayley}, because we %show that $\ell_{k}(G)$ is less than $\ord{G}$ for our groups of the form $\Z_n\rtimes_{\theta} \Z_m$.}, 
%  the hypotheses of \reth{main} do not apply, and we show that it  embeds in $B_{9}/\Gamma_{2}(P_{9})$. This shows that the hypotheses of this theorem are sufficient to embed $\Z_n\rtimes_{\theta} \Z_m$ in $B_n/\Gamma_k(P_n)$, but not necessary. The other one has not the format and also embeds in   $B_{9}/\Gamma_{2}(P_{9})$. 
 %We  also 
%give an example which is not of this form where the group is non Abelian and has  order $27$???????.

%The aim is to show that for certain values of $m$ and $n$, the group $G$ embeds in $B_n/\Gamma_k(P_n)$, the main result being \reth{main}. \comj{add comments about \resec{further}.}

\subsection{Proof of \reth{main}}\label{sec:proofmain}

%\comj{This paragraph seems to be a bit out of place here. We do not use the notation or mention the problem %elsewhere. Perhaps we could either put it at the end of the introduction, or at the end of the paper?} 
%\comj{or perhaps `The integer $\ell_k(G)$ is not always defined. For example, if $G$ is of even order, by \reth{torsionGamma23}(\ref{it:torsionGamma23a}), $G$ does not embed in any group of the form $B_{n}/\Gamma_k(P_{n})$.'}{\bf !!!DO!!! Ok}. 
%If $\ell_k(G)$ is defined, then \repr{tf} implies that $m(G)\leq \ell_k(G)$. 
%    %.\ Certainly these two numbers are not equal in general, for example, take $k=2$ and $G$ a $2$-group. Nevertheles 
% %   Under the hypothesis that $G$ has order odd and $k=2$, we do not know if these two numbers 
% %   $m(G)$ and $\ell_2(G)$ are the same. 
%  % is a lower bound for the integers $\ell$ for which $G$ embeds in 
%%    $B_{\ell}/\Gamma_2(P_n)$, as well for    $B_{\ell}/\Gamma_k(P_n)$ for any $k>2$. 
%We do not currently know whether this inequality is strict or not.
%Unlikely these two lower bou
%This implies that  $\ell=n$ is the least possible integer so that the group $G$ can 
  %be embedded in $B_{\ell}/\Gamma_k(P_{\ell})$.  So in this sense our results are the best as possible. 

Let $m,n\in \N$. In this section, $G$ will be a group of the form $\Z_n\rtimes_{\theta} \Z_m$. We study the question of whether $G$ embeds in $B_n/\Gamma_k(P_n)$, where $k\in \brak{2,3}$. %(resp.\ in $B_n/\Gamma_3(P_n)$). 
By~\reth{torsionGamma23}, when $G=\Z_n\rtimes_{\theta} \Z_m$ for such an embedding to exist, $\gcd(\ord{G},k!)=1$, and so we shall assume from now on that this is the case. 
%In all of this section, we shall assume that the representation $\map{\theta}{\Z_m}[\aut{\Z_n}]$ is injective.
In order to apply \repr{embedgen}, we will make use of a specific embedding of $G$ in $\sn$ studied by Marin in the case where $n$ is prime and $m=(n-1)/2$~\cite{Ma}, as well as 
  %Clearly the same formula can be used to provide an embedding of the groups $G$ as above, see below.  
the restriction to $G$ of the action of $\sn$ on $P_n/\Gamma_2(P_n)$ and $\Gamma_2(P_n)/\Gamma_3(P_n)$  described by \req{actionSn}. If $q\in \N$, we will denote the image of an integer $r$ under the canonical projection $\Z\to \Z_{q}$ by $r_{q}$, or simply by $r$ if no confusion is possible.
Following~\cite[proof of Corollary~3.11]{Ma}, we start by describing a homomorphism from $K$ to $\sn[A]$, for groups of the form $K=A\rtimes_{\theta} H$, where $A$ and $H$ are finite, $A$ is Abelian, and $\sn[A]$ denotes the symmetric group on the set $A$. Let $(u,v)\in K$, where the elements of $K$ are written with respect to the semi-direct product $A\rtimes_{\theta} H$, and let $\map{\varphi_{(u,v)}}{A}$ be the affine transformation defined by:
\begin{equation}\label{eq:phi}
\text{$\varphi_{(u,v)}(z)=\theta(v)(z)+u$ for all $z\in A$}.
\end{equation}
%and let $\sn[\Z_{n}]$ denote the symmetric group on the set $\Z_{n}$.

%\comj{We combined two lemmas into the following lemma.}
\begin{lem}\label{lem:Ginjphi}\mbox{}
\begin{enumerate}[(a)]
\item\label{it:Ginjphia} For all $(u,v)\in K$, the map $\map{\varphi_{(u,v)}}{A}$ defined in~\reqref{phi} is a bijection. 
\item\label{it:Ginjphib} Let $\map{\varphi}{K}[{\sn[A]}]$ be the map defined by $\varphi(u,v)=\varphi_{(u,v)}$ for all $(u,v)\in K$. Then $\phi$ is a homomorphism. If the action $\map{\theta}{H}[\aut{A}]$ is injective then $\phi$ is too.
\end{enumerate}
%For all $(u,v)\in G$, the map $\map{\varphi_{(u,v)}}{\Z_{n}}$ defined in~\reqref{phi} is a bijection. Further, if the action $\map{\theta}{\Z_{m}}[\aut{\Z_{n}}]$ is injective then the map $\map{\varphi}{G}[{\sn[\Z_{n}]}]$ defined by $\varphi(u,v)=\varphi_{(u,v)}$ is an injective homomorphism.
%if $\sn[\Z_{n}]$ denotes the symmetric group of the set $\Z_{n}$, the map $\map{\varphi}{G}[{\sn[\Z_{n}]}]$ defined by $\varphi(u,v)=\varphi_{(u,v)}$ is an injective homomorphism {\bf !!!D!!! } if $\theta$ is injective.
\end{lem}

\begin{proof}\mbox{}
\begin{enumerate}
\item If $(u,v)\in K$, the statement follows from the fact that $\theta(v)$ is an automorphism of $A$.
\item Part~(\ref{it:Ginjphia}) implies that the map $\varphi$ is well defined. We now prove that $\varphi$ is a homomorphism. If $(u_1,v_1),(u_2,v_2)\in K$, then for all $z\in A$, we have: 
%\comj{in these equations, probably $v_2v_1$ should be $v_2+v_1$?}
%The operation on the
% group of permutations of the set $S=\Z_n$ is given by  composition. So it suffices to show that 
% $\varphi_{(u_2,v_2)}\circ \varphi_{(u_1,v_1)}=\varphi_{(u_2+\theta(v_2)(u_1), v_2v_1)}$.  In fact  
\begin{align*}
(\varphi_{(u_2,v_2)}\circ \varphi_{(u_1,v_1)})(z)&=\varphi_{(u_2,v_2)}( \theta(v_1)(z)+u_1)= \theta(v_2)(\theta(v_1)(z)+u_1)+u_2\\
&= \theta(v_2)(\theta(v_1)(z))+\theta(v_2)(u_1)+u_2\\
&= \theta(v_2 v_1)(z)+\theta(v_2)(u_1)+u_2\\
&= \varphi_{(u_2+\theta(v_2)(u_1), v_2v_1)}(z) = \varphi_{(u_2,v_{2})(u_{1},v_{1})}(z),
\end{align*}
so $\varphi_{(u_2,v_2)}\circ \varphi_{(u_1,v_1)}=\varphi_{(u_2,v_{2})(u_{1},v_{1})}$, and $\varphi$ is a homomorphism. 
%\comj{Missing injectivity?}\como{Yes, unfortunately we forget it} \comj{Here is a proposal: 
Finally, suppose that $\theta$ is injective, and let $(u,v)\in \ker{\varphi}$. Then $z=\varphi(u,v)(z)=\varphi_{(u,v)}(z)=\theta(v)(z)+u$ for all $z\in A$. Taking $z$ to be the trivial element $e_A$ of $A$ yields $u=e_A$. Hence $\theta(v)=\id_{A}$, and  it follows that $v$ is the trivial element $H$ by the injectivity of $\theta$, which completes the proof of the lemma.\qedhere
\end{enumerate}
\end{proof}

As above, we consider the group $G=\Z_{n} \rtimes_{\theta} \Z_{m}$, where $\map{\theta}{\Z_{m}}[\aut{\Z_{n}}]$ is the associated action. Note that we can apply the construction of~\req{phi} to $G$, and so the conclusions of \relem{Ginjphi} hold for $G$. The element $1_m$ generates the additive group $\Z_m$, so $\theta(1_m)$ is an automorphism of $\Z_{n}$ whose order divides $m$, and since any automorphism of $\Z_n$ is multiplication by an integer that is relatively prime with $n$, there exists $1\leq t<n$ such that $\gcd(t,n)=1$, $\theta(1_m)$ is multiplication by $t$, and $t^m\equiv 1 \bmod{n}$.
%and so under the injectivity hypothesis on $\theta$, 
%so $\theta(1_m)$ is an automorphism of $\Z_{n}$ whose order divides $m$. Since any automorphism of $\Z_n$ is multiplication by an integer that is relatively prime with $n$, there exists $1\leq t<n$ such that $\gcd(t,n)=1$, $\theta(1_m)$ is multiplication by $t$, $t^m\equiv 1 \bmod{n}$.  
If $\theta$ is injective, the order of the automorphism $\theta(1_m)$ is equal to $m$, and so $t^{l}\not\equiv 1 \bmod{n}$ for all $1\leq l<m$, 
%Given a representation $\phi: Z_m\to Aut(Z_n)$ then the automorphism $\phi(1_m)$ is multiplication by an integer $t$ %relatively prime to $n$.
but this does not imply that $t^{l}-1$ is relatively prime with $n$. However, the condition that $\gcd(t^{l}-1,n)=1$ for all $1\leq l<m$ is the hypothesis that we require in order to prove \reth{main}, and as we shall now see, implies that $\theta$ is injective.

\begin{lem}\label{lem:fundamental} 
%\comj{The statement rewritten somewhat:}\como{OK} 
%\comj{DJ We have rewritten the statement. There were four items, but we think that only the first two items are necessary, and they have been combined into one statement which NOW is item b). Then we added the new item a).}
Let $n,m\in \N$, let $G$ be a semi-direct product of the form $\Z_{n} \rtimes_{\theta} \Z_{m}$, and let $1\leq t<n$ be such that $\theta(1_{m})$ is multiplication in $\Z_{n}$ by $t$. Suppose that $\gcd(t^l-1,n)=1$ for all $1\leq l \leq m-1$.
\begin{enumerate}[(a)]
\item\label{it:funda} The action $\map{\theta}{\Z_{m}}[\aut{\Z_{n}}]$ is injective, and the homomorphism $\map{\phi}{G}[{\sn[\Z_{n}]}]$ defined in \relem{Ginjphi} is injective. 
\item\label{it:fundb} For all $(u,v)\in G\setminus\brak{(0_{n},0_{m})}$, the permutation $\phi(u,v)$ fixes at most one element, and if $v\neq 0_{m}$ then $\phi(u,v)$ fixes precisely one element.
\end{enumerate}

\end{lem}

\begin{proof}\mbox{}
\begin{enumerate}[(a)]
\item To prove the first part, we argue by contraposition. Suppose that $\theta$ is not injective. Then there exists $1\leq l<m$ such that $\theta(l_{m})=\id_{\Z_{n}}$. Now $\theta(1_{m})$ is multiplication by $t$, so $\theta(l_{m})$ is multiplication by $t^{l}$, and thus $t^{l} \equiv 1 \bmod{n}$, which implies that $\gcd(t^{l}-1,n)\neq 1$. The second part of the statement follows from \relem{Ginjphi}(\ref{it:Ginjphib}).

\item Let
%\comj{I rewrote some of what follows to prove part~(\ref{it:fl0a}), but something seems to be missing, I think. In the %previous version, it was written that $(u, l)$ is an arbitrary element of $G$. This also includes the case where $l=0$. In %this case, the hypotheses do not apply. Shouldn't we also consider this case?} \comj{There is another proof of %part~(\ref{it:fl0a}) below\ldots? (from `Let $\alpha$ be an element\ldots' onwards.)}
 %{\bf !!!D!!! I correct a possible misprint}  
$(u, v)\in G\setminus \brak{(0_{n},0_{m})}$. 
% and consider the respective affine transformation $\varphi_{(u,l)}$ \comdo{This added:} given by \text{$\varphi_{(u,l)}(z)=t^lz+u$ for all $z\in \Z_{n}$}, since $\theta(l)=t^l$. 
If $v=0_{m}$ then $u\neq 0_{n}$, so $\phi_{(u,0_{m})}(z)=z+u\neq z$ for all $z\in \Z_{n}$, hence $\phi_{(u,0_{m})}$ is fixed-point free. So suppose that $v\neq 0_{m}$. Since $\theta(v)$ is multiplication by $t^{v}$, the corresponding affine transformation $\map{\varphi_{(u,v)}}{\Z_{n}}$ is given by $\varphi_{(u,v)}(z)=t^vz+u$ for all $z\in \Z_{n}$. By hypothesis, $\gcd(t^v-1,n)=1$, so $t^{v}-1$ is invertible in $\Z_{n}$, and if $z\in \Z_{n}$,  we have:
\begin{equation*}
\varphi_{(u,v)}(z)=z \Longleftrightarrow t^v z+u=z \Longleftrightarrow (t^{v}-1)z=-u \Longleftrightarrow z=-u(t^{v}-1)^{-1}
\end{equation*}
in $\Z_{n}$.
Hence $\varphi_{(u,v)}$ possesses a unique fixed point.\qedhere
\end{enumerate}
\end{proof}

%{\bf !!!DO!!! New Remark}
%
%
%\begin{rem}\label{gen}     The above  Lemma \ref{lem:fundamental} holds if we replace $Z_n$ by a finite Abelian group and $Z_m$ any finite group, and we keep the 
%same hypothesis for $\theta$.
%\end{rem}  

The above framework enables us to prove \reth{main} using \repr{embedgen}.

% enables us to apply \repr{embedgen} in order to . 
%it is easy to show that if a group of the form  $G=\Z_n\rtimes_{\theta} \Z_m$ which  satisfies 
%the hypothesis of the Lemma above, it can be embedded into $B_p/\Gamma_2(P_p)$.

\begin{proof}[Proof of \reth{main}]  %\comj{DJ New proof after the modification of Lemma 11}
Let $\map{\varphi}{G}[{\sn[\Z_{n}]}]$ be the embedding of $G$ in $\sn[\Z_{n}]$ of \relem{Ginjphi}(\ref{it:Ginjphib}). By \relem{fundamental}(\ref{it:fundb}), for all $g\in G\setminus \brak{e}$, the cycle decomposition of the permutation $\varphi(g)$ contains at most one fixed point.
So the embedding $\phi$ satisfies the hypotheses of \repr{embedgen}, from which we conclude that $G$ embeds in $B_n/\Gamma_2(P_n)$ (resp.\ in $B_n/\Gamma_3(P_n)$) if $mn$ is odd (resp.\ if $\gcd(mn,6)=1$).
\end{proof}

As an application of \reth{main}, we consider groups of the form $\Z_n\rtimes_{\theta} H$, where $H$ is finite, $n=p^r$ is a power of an odd prime $p$, where $r\in \N$, and the homomorphism $\map{\theta}{H}[\aut{\Z_{p^{r}}}]$ is injective.
% If $H$ is a subgroup of $\aut{\Z_n}$, $\Z_n\rtimes_{\theta'}H$ is a subgroup of $\Z_n\rtimes_{\theta} \aut{\Z_n}$, %where $\theta'$ is the restriction of $\theta$ to $H$. 
%If $n=p^r$ is a power of an odd prime $p$, where $r\in \N$, r
Recall from~\cite[p.~146, lines~16--17]{Za} that:
\begin{equation}\label{eq:decompZpr}
\aut{\Z_{p^r}}\cong \Z_{(p-1)p^{r-1}}\cong \Z_{p-1}\oplus \Z_{p^{r-1}},
\end{equation}
where the isomorphisms of~\reqref{decompZpr} are described in~\cite[pp.~145--146]{Za}. We now prove \reco{app} by studying  injective actions $\map{\theta}{H}[\Z_{(p-1)p^{r-1}}]$, where $H$ is a cyclic group whose order is an odd divisor of $p-1$.
%So $G$ is isomorphic to $\Z_{p^r}\rtimes_{\theta} \Z_{(p-1)p^{r-1}}$ where $\theta$ is obtained via the isomorphism %given by~\reqref{decompZpr} \comj{Isn't $\theta$ the identity?}. By considering certain subgroups of $G$, 

\begin{proof}[Proof of \reco{app}]
Let $p>2$ be prime, let $p-1=2^jd$, where $d$ is odd, and let $d_{1}$ be a divisor of $d$. So identifying $\aut{\Z_{p^{r}}}$ with $\Z_{p-1}\oplus \Z_{p^{r-1}}$ via~\reqref{decompZpr}, 
%\comj{The following rewritten slightly:} 
and taking the subgroup $H$ in the above discussion to be $\Z_{d_1}$,  there exists an injective homomorphism $\map{\theta}{\Z_{d_1}}[\Z_{p-1}\oplus \Z_{p^{r-1}}]$, where $\theta(1_{d_1})$ is an automorphism of $\Z_{p^{r}}$ given by multiplication by an integer $t$ that is relatively prime with $p$, and the order $d_{1}$ of this automorphism is also relatively prime with $p$. In particular, the order of $t$ in the group $\Z_{p^r}^{\ast}$ is equal to $d_{1}$. We claim that $\gcd(t^{l}-1,p)=1$ for all $0<l<d_1$. Suppose on the contrary that
 % the order of the automorphism of $\Z_{p^r}$ given by multiplication by $t^{l}$ is relatively prime with $p$.
$t^l-1$ is divisible by $p$ for some $0<l<d_1$. Then $t^{l}=1+kp$, where $k\in \N$, and~\cite[p.~146, line~12]{Za} implies that the order of $t^{l}$ in $\Z_{p^r}^{\ast}$ is a power of $p$, which contradicts the fact that $t$ is of order $d_{1}$ and $\gcd(d_{1},p)=1$. This proves the claim. 
%Hence $\gcd(t^{l}-1,p)=1$ for all $1\leq l\leq d_1-1$, 
Part~(\ref{it:appa}) (resp.\ part~(\ref{it:appb})) follows using the fact that the order of $\Z_{p^{r}}\rtimes_{\theta} \Z_{d_{1}}$ is odd (resp.\ relatively prime with $6$) and by applying \reth{main}.
%and we conclude from \reth{main} that $\Z_{p^{r}}\rtimes_{\theta} \Z_{d_{1}}$ embeds in $B_{p^{r}}/\Gamma_{2}(P_{p^{r}})$ (resp.\ in $B_{p^{r}}/\Gamma_{3}(P_{p^{r}})$).
\end{proof}

\begin{rem}%\mbox{}
%\begin{enumerate}\item 
As we mentioned in the introduction, the results of \reco{app}  are sharp in the sense that if $k\in\brak{2,3}$, the groups $\Z_{p^r}\rtimes_{\theta} \Z_d$ satisfy $\ell_{k}(G)=m(G)=p^{r}$ if $d>1$, where $k\in \brak{2,3}$. 
%given in the statement cannot embed in $B_{m}/\Gamma_k(P_{m})$ for any $m<p^r$. This follows from \repr{tf}(\ref{it:ptorsiona}) and the fact that the group $\Z_{p^r}$ cannot be embedded in $\sn[m]$ for any $m<p^r$. 
%%\comj{This added:} 
%So with respect to the inequalities of~\reqref{ineqmG}, $m(G)=\ell_{k}(G)<\lvert G\rvert$ for such a group $G$ for which $d>1$. 
%It is unlikely that this will be the case in general, \comj{check this using~\cite[Corollary~13]{MaV} in the case $k=2$}
%\comj{This reformulated a little:} 
%\comj{This changed:!!!DO!!! Looks good} 
This result no longer holds if we remove the hypothesis that the order of the group being acted upon is a prime power. For example, if in the semi-direct product $G=\Z_{n} \rtimes \Z_{d}$, we take $n=15$ and $d=1$ then $G\cong \Z_{3}\times \Z_{5}$, and $\ell_{2}(G)=m(G)=8<15$ using~\cite[Theorem~3(b)]{GGO1} or~\cite[Corollary~13]{MaV}.
%This is not true in general. For example, if $d=1$ in other words, if $n$ is not a prime power, it may be possible to embed groups of the form $\Z_n\rtimes_{\theta} H$ in $B_{r}/\Gamma_k(P_{r})$ for some $r<n$. 
% We first observe that in order to have $G$ embedded in   some   
%  $B_{\ell}/\Gamma_k(P_{\ell})$, it is necessary that we have $G$ embedded in $\sn$. 
 % In our  case  possibly the group  $G$ may embeds in $\mathcal{S}_{\ell}$ for $\ell<n$.
%\item The groups that appear in~\cite[Corollary 3.11]{Ma} correspond to the case where $r=1$, $p\equiv 3\bmod{4}$, and $d_1=(p-1)/2$ is odd. So \reco{app} generalises~\cite[Corollary 3.11]{Ma} to the case where $p$ is any odd prime by taking $d_{1}$ 
%%\comj{$d_{1}$?} 
%to be the greatest odd divisor of $p-1$.
%integer of the form $(p-1)/2^s$.
%If $p-1$ is divisible by $4$, the results of  \reco{app} extend those of~\cite[Corollary 3.11]{Ma}
% result above \comj{??} by taking $m=(p-1)/2^s$ so that $m$ is an odd integer.
%\end{enumerate}
\end{rem}

\subsection{Further examples}\label{sec:further}

%\comj{Rewritten slightly:}
In this final section, we give examples of two semi-direct products of the form $\Z_{9}\rtimes_{\theta} \Z_3$ and $(\Z_{3}\oplus \Z_{3})\rtimes_{\theta} \Z_3$ respectively that do not satisfy the hypotheses of  \reth{main}, but that embed in $B_9/\Gamma_2(P_9)$. 
%\comj{This added as motivation:} 
%Within our framework, it is natural to study these groups, first because they are of order $27$, so are related to the discussion in \resec{proofmain} on groups whose order is a prime power, and secondly because with the exception of the Frobenius group of order $21$ that was studied in~\cite{GGO1}, they are the smallest non-Abelian groups of odd order.
We start with some general comments. If $p$ is an odd prime, consider the group $G=\Z_{p^r}\rtimes_{\theta} \Z_{p^{r-1}(p-1)}$, where with respect to the notation of~\cite[p.~146, line~8]{Za}, the homomorphism $\map{\theta}{\Z_{p^{r-1}(p-1)}}[\aut{\Z_{p^r}}]$ sends $1_{p^{r-1}(p-1)}$ to the element of $\aut{\Z_{p^r}}$ given by multiplication in $\Z_{p^{r}}$ by $t'$, where $t'=(1+p)g_1$, and where the order of $g_1$ (resp.\ $1+p$) is equal to $p-1$ (resp.\ $p^{r-1}$)
% \comj{check this} 
in the multiplicative group $\Z_{p^r}^{\ast}$. If $d_{1}$ is an odd divisor of $p^{r-1}(p-1)$, let $q=p^{r-1}(p-1)/d_1$, and consider the subgroup $\Z_{p^r}\rtimes_{\theta'} \Z_{d_1}$ of $G$, where $\Z_{d_1}$ is the subgroup of $\Z_{p^{r-1}(p-1)}$ of order $d_{1}$, and $\map{\theta'}{\Z_{d_1}}[\aut{\Z_{p^r}}]$ is the restriction of $\theta$ to $\Z_{d_1}$. Then $\theta'(1_{d})=\theta(q_{p^{r-1}(p-1)})$ is multiplication by $t=t'^{q}$ in $\Z_{p^{r}}$,
%\comj{this added:} 
and by injectivity, $t$ is of order $d_{1}$ in $\Z_{p^{r}}^{\ast}$. If further $d_{1}$ is divisible by $p$ then $t^{d_{1}/p}$ is of order $p$ in $\Z_{p^{r}}^{\ast}$, and by \cite[p.~146, line~12]{Za}, $t^{d_{1}/p}\equiv (1+p)^{\lambda {p^{r-2}}}$ mod $p^r$, where $\lambda\in \N$ and $\gcd(\lambda, p)=1$. It follows that $t^{d_{1}/p}-1$ is divisible by $p$, and since $0<d_{1}/p<d_{1}$, the hypotheses of \reth{main} are not satisfied for the group $\Z_{p^r}\rtimes_{\theta'} \Z_{d_1}$. As Example~\ref{ex:groups27}(\ref{it:exama}) below shows, if $p=3$ and $r=2$, such a group may nevertheless embed in $B_{p^r}/\Gamma_2(P_{p^r})$. In Example~\ref{ex:groups27}(\ref{it:examb}), we show that the other non-Abelian group of order $27$, which is of the form $(\Z_{p}\oplus \Z_{p})\rtimes \Z_{p}$, also embeds in $B_9/\Gamma_2(P_9)$. It does not satisfy the hypotheses of \reth{main} either.

%where $\Z_{d_1}$ is the subgroup generated by $t^{q}$, with $q=\frac{p^{r-1}(p-1)}{d_1}$, and $d_1$  divides  $p^{r-1}(p-1)$ with $d_1$ odd. If $d_1$ is divisible by $p$  then $(t^q)^{q_1}$ has order $p$ for some $q_1$, and by \cite[p.~146, line~12]{Za} $(t^q)^{q_1}\equiv (1+p)^{\lambda {p^{n-2}}}$ mod $p^n$, with $\gcd(\lambda, p)=1$. 
%Hence $(t^q)^{q_1}-1$ is divisible by $p$, so the hypotheses of \relem{fundamental} are not satisfied. But this does not necessarily imply that this type of group cannot be embedded in $B_{p^r}/\Gamma_2(P_{p^r})$. \comj{Perhaps say how this relates to the examples that follow? YES JUST ADD ABOUT a)and also two words about b) (not in the family above} In fact the first example below it is a finite group of the form  
%    $\Z_{p^r}\rtimes_{\theta_{\mid}} \Z_{d_1}$ where the   hypotheses of \relem{fundamental} are not satisfied, but it can 
%    be embbeded in  $B_{p^r}/\Gamma_2(P_{p^r})$. The second example is a non-abelian group which is not of the form 
%   $\Z_{p^r}\rtimes_{\theta} \Z_{p^{r-1}(p-1)}$. 

\begin{exms}\mbox{}\label{ex:groups27}
Suppose that $G$ is a group of order $27$ that embeds in $\sn[9]$. If $k=2$ and $n=9$ in~\reqref{sesgammak}, taking the preimage of $G$ by $\overline{\sigma}$ leads to the following short exact sequence:
\begin{equation}\label{eq:sesB9}
1 \to P_{9}/\Gamma_{2}(P_{9}) \to \overline{\sigma}^{-1}(G) \xrightarrow{\overline{\sigma}\left\lvert_{\overline{\sigma}^{-1}(G)}\right.} G \to 1,
\end{equation}
where the rank of the free Abelian group $P_{9}/\Gamma_{2}(P_{9})$ is equal to $36$. As we shall now see, if $G$ is one of the two non-Abelian groups of order $27$, the action of $\sn[9]$ on the basis $\mathcal{B}$ of $P_{9}/\Gamma_{2}(P_{9})$ given by~\reqref{actionSn} restricts to an action of $G$ on $\mathcal{B}$ for which there are two orbits, that of $A_{1,2}$, which contains $9$ elements, and is given by:
\begin{equation}\label{eq:orbA12}
\mathcal{O}=\brak{A_{1,2},A_{8,9},A_{5,6},A_{2,3},A_{7,8},A_{4,5},A_{1,3},A_{7,9},A_{4,6}},
\end{equation}
and that of $A_{5,9}$, which contains the remaining $27$ elements of $\mathcal{B}$. Let $H$ be the subgroup of $P_{9}/\Gamma_{2}(P_{9})$ generated by the orbit of $A_{5,9}$. Then $H\cong \Z^{27}$, and $H$ is not normal in $B_{9}/\Gamma_{2}(P_{9})$, but it is normal in the subgroup $\overline{\sigma}^{-1}(G)$ of $B_{9}/\Gamma_{2}(P_{9})$ %\comj{This added:} 
since the basis $\mathcal{B}\setminus \mathcal{O}$ of $H$ is invariant under the action of $G$.
%\comj{justify $H \lhd \overline{\sigma}^{-1}(G)$?} \comd{but it is normal in the subgroup $\overline{\sigma}^{-1}(G)$ of $B_{9}/\Gamma_{2}(P_{9})$, since any element of  $\overline{\sigma}^{-1}(G)$ maps the  elements of $\mathcal{B}$(a base of $H$) into itself.}. %If $\widetilde{H}$ denotes the quotient of $\overline{\sigma}^{-1}(G)$ by $H$. 
We thus have an extension: 
%\comj{The following is a general argument I think, and makes use of parts that were originally in the two examples. See %what you think.}
\begin{equation}\label{eq:quotH}
1\to H \to \overline{\sigma}^{-1}(G) \stackrel{\pi}{\to} \widetilde{H}\to 1,
\end{equation}
where $\widetilde{H}=\overline{\sigma}^{-1}(G)/H$, and $\map{\pi}{\overline{\sigma}^{-1}(G)}[\widetilde{H}]$ is the canonical projection. Equation~\reqref{sesB9} induces the following short exact sequence:
\begin{equation}\label{eq:sesB9quot}
1 \to (P_{9}/\Gamma_{2}(P_{9}))/H \to \widetilde{H} \stackrel{\widetilde{\sigma}}{\to} G \to 1,
\end{equation}
where $\map{\widetilde{\sigma}}{\widetilde{H}}[G]$ is the surjective homomorphism induced by $\overline{\sigma}$, 
% \comj{This added:} 
and we also have an extension:
\begin{equation}\label{eq:quotH2}
1\to H \to P_{9}/\Gamma_{2}(P_{9}) \stackrel{\rho}{\to} (P_{9}/\Gamma_{2}(P_{9}))/H \to 1,
\end{equation}
obtained from the canonical projection $\map{\rho}{P_{9}/\Gamma_{2}(P_{9})}[(P_{9}/\Gamma_{2}(P_{9}))/H]$. From the construction of $H$, and using the fact that $P_{9}/\Gamma_{2}(P_{9})$ (resp.\ $H$) is the free Abelian group of rank $36$ (resp.\ $27$) for which $\mathcal{B}$ (resp.\ $\mathcal{B}\setminus \mathcal{O}$) is a basis, the kernel of~\reqref{sesB9quot} is isomorphic to $\Z^{9}$ and a basis is given by the $H$-cosets of the nine elements of $\mathcal{O}$. 
%\comj{This added:} 
Further, the restriction of $\rho$ to the subgroup of $P_{9}/\Gamma_{2}(P_{9})$ generated by $\mathcal{O}$ is an isomorphism, and thus the set $\rho(\mathcal{O})$ is a basis of $(P_{9}/\Gamma_{2}(P_{9}))/H$, which we denote by $\mathcal{O}'$. In the examples below, we shall construct an explicit embedding $\map{\iota}{G}[\widetilde{H}]$ of $G$ in $\widetilde{H}$. This being the case, using~\reqref{quotH}, we thus obtain the following short exact sequence: 
\begin{equation}\label{eq:quotiotaG}
1\to H \to \pi^{-1}(\iota(G)) \xrightarrow{\pi\left\lvert_{\pi^{-1}(\iota(G))}\right.} \iota(G)\to 1.
\end{equation}
Now $H$ is isomorphic to $\Z[G]$, and the action of $\iota(G)$  on $H$ is given via~\reqref{actionSn}.
% and the embedding of $G$ in $\widetilde{H}$ \comj{Something is missing in this sentence}. 
It follows from \repr{null} that~\reqref{quotiotaG} splits. Hence $G$ embeds in $\pi^{-1}(\iota(G))$, which is a subgroup of $\overline{\sigma}^{-1}(G)$, and so it embeds in $B_{9}/\Gamma_{2}(P_{9})$. We could have attempted to embed $G$ in $B_{9}/\Gamma_{2}(P_{9})$ directly via~\reqref{sesB9}. However one of the difficulties with this approach is that the rank of the kernel is $36$, whereas that of the kernel of~\reqref{sesB9quot} is much smaller, and this decreases greatly the number of calculations needed to show that $G$ embeds in $\widetilde{H}$. We now give the details of the computations of this embedding in the two cases.
%\comd{It follows from \repr{null} that~\reqref{quotiotaG} splits. Hence $G$ embeds in $\pi^{-1}(\iota(G))$, which is a subgroup of $\overline{\sigma}^{-1}(G)$, and so it embeds in $B_{9}/\Gamma_{2}(P_{9})$. We could have attempted to embed $G$ in $B_{9}/\Gamma_{2}(P_{9})$ directly via~\reqref{sesB9}. However one of the difficulties with this approach is that the rank of the kernel is $36$, whereas that of the kernel of~\reqref{sesB9quot} is much smaller, and this decreases greatly the number of calculations needed to show that $G$ embeds in $\widetilde{H}$. We now give the details of the computations of this embedding in the two cases.} 
%\comj{The two examples have been rewritten somewhat, taking into account the fact that some of the more general material is above. }

\begin{enumerate}
\item\label{it:exama}
%Nevertheless  consider the group $\Z_{p^2}\rtimes_{\theta} \Z_p$ where
%$\theta(1_p)$ is the automorphism multiplication by
%$(1+p)$. This group can be embedded in $\sn[p^2]$ using the function
%$\varphi$ as defined in \req{phi}. To exemplify, let $p=3$ and  
Consider the semi-direct product $\Z_{9}\rtimes_{\theta'} \Z_3$, where $\map{\theta'}{\Z_{3}}[\aut{\Z_{9}}]$ is as defined in the second paragraph of this subsection. We have $g_{1}=8$, so $\theta'(1_3)$ is multiplication by $(1+3)\ldotp (8)=32\equiv 5 \bmod{9}$. 
%\comj{$5$?}\bmod{9}$.
 It will be more convenient for us to work with a different automorphism, but that gives rise to a semi-direct product that is isomorphic to $\Z_{9}\rtimes_{\theta'} \Z_3$ as follows. Let $\map{\phi}{\Z_9}$ be the automorphism given by multiplication by $4$, let $\map{\theta}{\Z_{3}}[\aut{\Z_{9}}]$ be such that $\theta(1_3)=\phi$, and let $G=\Z_{9}\rtimes_{\theta} \Z_3$. Then 
%$\theta_1=\phi \circ \theta$ \comj{how is this defined? Is this 
$\theta'(1_{3})=\phi \circ \theta(1_{3})$, which implies that the groups $G$ and $\Z_{9}\rtimes_{\theta'} \Z_3$ are isomorphic, see~\cite[Chap.~1.2,  Proposition~12]{AB}.
%    \comj{In another paper, I  think we found a reference once for this fact\ldots}. 
%We  will  work with $\theta$ which is more convenient for the purpose of our calculation.
% convinience of the calculation. 
%of $\Z_9$  homomorphism, and $\theta(1_{3})$ is multiplication by $4$. 
%\comj{This group is of the above form?} 
It is clear that $\theta$ is injective, but that the hypothesis of \reth{main} that $\gcd(t^l-1,n)=1$ for all $1\leq l<d$ is not satisfied if $l=1$. Nevertheless, the group $G$ embeds in $B_9/\Gamma_2(P_9)$. To see this, 
%, although the hypothesis of \relem{fundamental} are not satisfied.  
consider the elements $\alpha=(1,2,3)(4,5,6)(7)(8)(9)$ and $\beta=(1,4,7,3,5,8,2,6,9)$ of $\sn[9]$. Then: 
\begin{align*}
\alpha \beta\alpha^{-1}&=(1,2,3)(4,5,6)(7)(8)(9) \ldotp (1,4,7,3,5,8,2,6,9)\ldotp ((1,2,3)(4,5,6)(7)(8)(9))^{-1}\\
 &=(1,5,9,3,6,7,2,4,8)=(1,4,7,3,5,8,2,6,9)^4=\beta^{4},
\end{align*}
and an embedding of $G$ in $\sn[9]$ is realised by sending $1_{3}$ (resp.\ $1_{9}$) to $\alpha$ (resp.\ $\beta$). From now on, we identify $G$ with its image in $\sn[9]$ under this embedding.
%An embedding of $G$ in $\sn[9]$ is given by $1_{9} \mapsto \beta=(1,4,7,3,5,8,2,6,9)$ and $1_{3} \mapsto \alpha=(1,2,3)(4,5,6)(7)(8)(9)$.
%By direct calculation we verify that
%\begin{align*}
% (1,2,3)(4,5,6)(7)(8)(9) \circ (1,4,7,3,5,8,2,6,9)\circ &
%((1,2,3)(4,5,6)(7)(8)(9))^{-1}\\
% &=(1,5,9,3,6,7,2,4,8)\\
%&=(1,4,7,3,5,8,2,6,9)^4.
%\end{align*}
%We claim that there exists an embedding $\iota\colon\thinspace G\lhra B_9/\Gamma_2(P_9)$. In order to show that there exists this 
%We first show that $G$ embeds in a quotient of $B_9/\Gamma_2(P_9)$. 
%From now on, we identify $G$ with its image in $\sn[9]$. 
%\comj{I added a few details:}
%To obtain an embedding of $G$ in $B_9/\Gamma_2(P_9)$, we first show that $G$ embeds into a quotient of $B_9/\Gamma_2(P_9)$. 
One may check that $\mathcal{O}$ and $\mathcal{B}\setminus \mathcal{O}$ are the two orbits arising from the 
%The basis $\mathcal{B}$ of $P_9/\Gamma_2(P_9)$ given in~\reqref{defBBprime} contains $36$ elements, and
%The 
action of $G$ on $\mathcal{B}$ given by~\reqref{actionSn} (for future reference, note that the order of the elements of $\mathcal{O}$ is that obtained by the action of successive powers of $\beta$).
We define the map $\map{\iota}{G}[\widetilde{H}]$ on the generators of $G$ by $\iota(\alpha)= \widehat{\alpha}$ and $\iota(\beta)=\widehat{\beta}$, where $\widehat{\alpha}=\sigma_2 \sigma_1^{-1} \sigma_5\sigma_4^{-1}$ and $\widehat{\beta}=A_{1,2}A_{8,9}^{-1} w \sigma_8\sigma_7 \sigma_6 \sigma_5 \sigma_4^{-1}\sigma_3^{-1} \sigma_2^{-1} \sigma_1^{-1} w^{-1}$, and where $w=\sigma_3 \sigma_6 \sigma_2 \sigma_3 \sigma_4\sigma_5\sigma_4 \sigma_3 \sigma_7\sigma_6$. 
%We define the map $\map{\iota}{G}[\widetilde{H}]$ on the generators of $G$ by $\iota(\alpha)= \widehat{\alpha}$ and $\iota(\beta)=\widehat{\beta}$. 
By abuse of notation, we will denote an element of $B_9/\Gamma_2(P_9)$ in the same way as its projection on the quotient  $(B_{9}/\Gamma_{2}(P_{9}))/H$.
We claim that $\widehat{\alpha} \widehat{\beta} \widehat{\alpha}^{-1} \widehat{\beta}^{-4}=1$ in $\widetilde{H}$, from which we may conclude that $\iota$ extends to a group homomorphism. 
To prove the claim, using~\cite[Proposition~12]{GGO1} and the action of $\beta$ on the orbit of $A_{1,2}$ mentioned above, first note that:
\begin{align}
\widehat{\alpha} \widehat{\beta} \widehat{\alpha}^{-1} \widehat{\beta}^{-4}&= \widehat{\alpha} A_{1,2}A_{8,9}^{-1} b \widehat{\alpha}^{-1} (A_{1,2}A_{8,9}^{-1}b)^{-4}\notag\\
&=A_{1,3}A_{8,9}^{-1} \widehat{\alpha} b \widehat{\alpha}^{-1} b^{-4}\ldotp b^{3} A_{1,2}^{-1}A_{8,9} b^{-3} \ldotp b^{2} A_{1,2}^{-1}A_{8,9} b^{-2} \ldotp b A_{1,2}^{-1}A_{8,9} b^{-1} \ldotp A_{1,2}^{-1}A_{8,9}\notag\\
&= A_{1,3}A_{8,9}^{-1} \ldotp \widehat{\alpha} b \widehat{\alpha}^{-1} b^{-4} \ldotp A_{2,3}^{-1} A_{7,8} \ldotp A_{5,6}^{-1} A_{2,3} \ldotp A_{8,9}^{-1} A_{5,6} \ldotp A_{1,2}^{-1}A_{8,9}\notag\\
&=A_{1,2}^{-1} A_{7,8} A_{1,3}A_{8,9}^{-1} \widehat{\alpha} b \widehat{\alpha}^{-1} b^{-4}.\label{eq:relalbe}
\end{align}
To obtain the final equality, we have also used the fact that $\widehat{\alpha} b \widehat{\alpha}^{-1} b^{-4}$ belongs to the quotient $(P_{9}/\Gamma_{2}(P_{9}))/H$, so commutes with the other terms in the expression. To compute $\widehat{\alpha} b \widehat{\alpha}^{-1} b^{-4}$ in terms of the basis $\mathcal{O}'$ of $(P_9/\Gamma_2(P_9))/H$, we use the method of crossing numbers given in~\cite[Proposition~15]{GGO1}, except that since we are working in $(P_9/\Gamma_2(P_9))/H$, using the isomorphism $\map{\rho\left\lvert_{\ang{\mathcal{O}}}\right.}{\ang{\mathcal{O}}}[(P_{9}/\Gamma_{2}(P_{9}))/H]$ induced by~\reqref{quotH2}, it suffices to compute the crossing numbers of the pairs of strings corresponding to the elements of $\mathcal{O}$ given in~\reqref{orbA12}. Using the braid $\widehat{\alpha} b \widehat{\alpha}^{-1} b^{-4}$ illustrated in Figure~\ref{fig:fig3}, one may verify that:
%In this way, we may check that: \comj{Add some pictures?}\como{See Figures \ref{fig:braid_1_1} and \ref{fig:braid_1_2}. The product $\widehat{\alpha} b \widehat{\alpha}^{-1} w \cdot \delta_{0,9}^{-4}w^{-1}$ is equal to $\widehat{\alpha} b \widehat{\alpha}^{-1} b^{-4}$.
%Are there the desired pictures? I used the package called 'braids', it is possible to change colors and modify other figure options.}
\begin{equation}\label{eq:relsd27}
\widehat{\alpha} b \widehat{\alpha}^{-1}b^{-4}=A_{1,2}A_{1,3}^{-1}A_{7,8}^{-1}A_{8,9}
\end{equation}
in the quotient $(P_9/\Gamma_2(P_9))/H$.
\begin{figure}\enlargethispage{1cm}\vspace*{-10mm}
  \centering
  \begin{minipage}[b]{.48\linewidth}
    \centering
%    \subcaptionbox{First top left}
      {\begin{tikzpicture}
\braid[
%style all floors={fill=yellow},
% style floors={1}{dashed,fill=yellow!50!green},
%floor command={%
%\fill (\floorsx,\floorsy) rectangle (\floorex,\floorey);
%\draw (\floorsx,\floorsy) -- (\floorex,\floorsy);
%},
height=14.4pt, 
width=22pt, %Distance among strings
border height=1pt, %'Size' of the start and end of strings
line width=1.5pt,
style strands={1}{red},
style strands={2}{orange},
style strands={3}{DarkKhaki}, 
style strands={4}{green}, 
style strands={5}{blue}, 
style strands={7}{Fuchsia},
style strands={8}{violet},
style strands={9}{DarkBlue}
] 
s_{2}-s_{5} 
s_1^{-1}-s_4^{-1}-s_{6} 
s_{3}-s_{7} 
s_{2}-s_{8} 
s_{3} 
s_{4} 
s_{5} 
s_{4}-s_{6} 
s_{3}-s_{7} 
s_{6} 
s_{5} 
s_4^{-1}-s_6^{-1} 
s_{3}^{-1}-s_7^{-1} 
s_2^{-1} 
s_1^{-1}-s_3^{-1} 
s_4^{-1} 
s_5^{-1} 
s_4^{-1}-s_6^{-1} 
s_3^{-1} 
s_2^{-1} 
s_{1}-s_{3}^{-1} 
s_{2}^{-1}-s_{4}
s_{3}-s_5^{-1} 
s_{2}-s_{6} 
s_{1}-s_{3}-s_{7}
s_{4} 
s_{5} 
s_{4}-s_{6} 
s_{3}
s_{2} 
s_{1}-s_{3} 
s_{2}-s_{4} 
s_{1}-s_{3}-s_{5}^{-1} 
s_{2}-s_{4}-s_{6}^{-1} 
s_{1}-s_{3}-s_{5}^{-1}-s_{7}^{-1} 
s_{2}-s_{4}-s_{6}^{-1}-s_{8}^{-1} 
s_{3}-s_{5}^{-1}-s_{7}^{-1} 
s_{4}-s_{6}^{-1}-s_{8}^{-1} 
s_{3}^{-1}-s_{5}^{-1}-s_{7}^{-1} 
s_{4}^{-1}-s_{6}^{-1}-s_{8}^{-1} 
s_{7}^{-1} 
s_{6}^{-1}-s_{8}^{-1} 
s_{5}^{-1}-s_{7}^{-1} 
s_{4}^{-1}-s_{6}^{-1} 
s_{3}^{-1} 
s_{2}^{-1}
s_{3}^{-1}; 
\end{tikzpicture}}
\caption{The braid $\widehat{\alpha} b \widehat{\alpha}^{-1} b^{-4}$}\label{fig:fig3}
  \end{minipage}\quad
  \begin{minipage}[b]{.48\linewidth}
    \centering
    \subcaptionbox{The braid $[\widehat{\alpha}, \widehat{\beta}] \ldotp \widehat{\gamma}^{-1}$\label{fig:fig4a}}
      {\begin{tikzpicture}
\braid[
%style all floors={fill=yellow},
% style floors={1}{dashed,fill=yellow!50!green},
%floor command={%
%\fill (\floorsx,\floorsy) rectangle (\floorex,\floorey);
%\draw (\floorsx,\floorsy) -- (\floorex,\floorsy);
%},
height=12pt, 
width=22pt, %Distance among strings
border height=1pt, %'Size' of the start and end of strings
line width=1.5pt,
style strands={1}{red},
style strands={2}{orange},
style strands={3}{DarkKhaki}, 
style strands={4}{green}, 
style strands={5}{blue}, 
style strands={7}{Fuchsia},
style strands={8}{violet},
style strands={9}{DarkBlue}
]
s_{3}-s_{6} s_{2}-s_{4}-s_{7} s_{5}-s_{8} s_{4}-s_{6} s_{3}-s_{5}-s_{7}^{-1} s_{2}-s_{4}^{-1}-s_{6}^{-1} s_{1}-s_{3}^{-1}-s_{7}^{-1} s_{2}^{-1}-s_{4}^{-1} s_{5}^{-1} s_{4}^{-1}-s_{6}^{-1} s_{3}^{-1}-s_{5}-s_{7} s_{4}^{-1}-s_{6}-s_{8}^{-1} s_{3}^{-1}-s_{7} s_{2}-s_{4} s_{1}-s_{5} s_{4}-s_{6} s_{3}-s_{7} s_{2}^{-1}-s_{4}-s_{8}^{-1} s_{3}^{-1}-s_{5}^{-1} s_{2}^{-1}-s_{4}^{-1}-s_{6}^{-1} s_{1}-s_{5}^{-1}-s_{7}^{-1} s_{4}^{-1}-s_{6}^{-1}-s_{8} s_{3}^{-1}-s_{7}^{-1} s_{2}^{-1}-s_{4}-s_{7} s_{5}^{-1}-s_{8}^{-1} s_{4} s_{5}^{-1}
;
\end{tikzpicture}}\vspace*{3mm}

\subcaptionbox{The braid $[\widehat{\alpha}, \widehat{\gamma}]$\label{fig:fig4b}}
      {\begin{tikzpicture}
\braid[
%style all floors={fill=yellow},
% style floors={1}{dashed,fill=yellow!50!green},
%floor command={%
%\fill (\floorsx,\floorsy) rectangle (\floorex,\floorey);
%\draw (\floorsx,\floorsy) -- (\floorex,\floorsy);
%},
height=12pt, 
width=22pt, %Distance among strings
border height=1pt, %'Size' of the start and end of strings
line width=1.5pt,
style strands={1}{red},
style strands={2}{orange},
style strands={3}{DarkKhaki}, 
style strands={4}{green}, 
style strands={5}{blue}, 
style strands={7}{Fuchsia},
style strands={8}{violet},
style strands={9}{DarkBlue}
] 
s_{3}-s_{6} s_{2}-s_{4}-s_{7} s_{5}-s_{8} s_{4}-s_{6} s_{3}-s_{5}-s_{7}^{-1} s_{2}-s_{4}^{-1}-s_{6}^{-1} s_{1}^{-1}-s_{3}^{-1}-s_{7}^{-1} s_{2}^{-1}-s_{4}^{-1}-s_{8} s_{5}^{-1} s_{4}^{-1}-s_{6}^{-1} s_{3}^{-1}-s_{5}-s_{7}^{-1} s_{2}-s_{4}^{-1}-s_{6} s_{1}^{-1}-s_{3}-s_{7} s_{2}-s_{4} s_{1}-s_{5} s_{4}-s_{6} s_{3}-s_{7} s_{2}^{-1}-s_{4}-s_{8}^{-1} s_{3}^{-1}-s_{5}^{-1} s_{2}^{-1}-s_{4}^{-1}-s_{6}^{-1} s_{1}-s_{5}^{-1}-s_{7}^{-1} s_{4}^{-1}-s_{6}^{-1} s_{3}^{-1}-s_{7} s_{2}^{-1}-s_{4}-s_{8}^{-1} s_{5}^{-1}
;
\end{tikzpicture}}

    \caption{The braids $[\widehat{\alpha}, \widehat{\beta}] \ldotp \widehat{\gamma}^{-1}$ and $[\widehat{\alpha}, \widehat{\gamma}]$}\label{fig:fig4}
  \end{minipage}
\end{figure}
%*********************************************
%
%, where $b=w \sigma_8\sigma_7 \sigma_6\sigma_5 \sigma_4^{-1} \sigma_3^{-1} \sigma_2^{-1}\sigma_1^{-1} w^{-1}$. 
%$b=\sigma_3\sigma_6\sigma_2\sigma_3\sigma_4\sigma_5\sigma_4\sigma_3\sigma_7\sigma_6\cdot
%\sigma_8\sigma_7\sigma_6\sigma_5\sigma_4^{-1}\sigma_3^{-1}\sigma_2^{-1}\sigma_1^{-1}\cdot
%(\sigma_3\sigma_6\sigma_2\sigma_3\sigma_4\sigma_5\sigma_4\sigma_3\sigma_7\sigma_6)^{-1}$.
%Note that the expression of the right hand side in \req{relsd27} was obtained from the left hand side using crossing numbers of the strings of the pure braid $\widehat{\alpha} b \widehat{\alpha}^{-1}b^{-4}$~\cite[Proposition~15]{GGO1}. 
It follows from equations~\reqref{relalbe} and~\reqref{relsd27} that $\widehat{\alpha} \widehat{\beta} \widehat{\alpha}^{-1} \widehat{\beta}^{-4}=1$ in $\widetilde{H}$, which proves the claim. Thus $\langle\, \widehat{\alpha}, \widehat{\beta} \,\rangle$ is a quotient of $G$, but since it is non Abelian, and the only non-Abelian quotient of $G$ is itself, we conclude that $\langle\, \widehat{\alpha}, \widehat{\beta} \,\rangle \cong G$, and hence $\iota$ is an embedding. It follows from the discussion at the beginning of these examples that $G$ embeds in $\pi^{-1}(\iota(G))$, and therefore in $B_9/\Gamma_2(P_9)$.
%From the expression of $\widehat{\beta}=A_{1,2}A_{8,9}^{-1}b$, the formula for the conjugacy action of $B_9/\Gamma_2(P_9)$ on the generators of $P_9/\Gamma_2(P_9)$ given in~\cite[Proposition~12]{GGO1}  and \req{relsd27} we show that $\widehat{\alpha} \widehat{\beta} \widehat{\alpha}^{-1} \widehat{\beta}^{-4}=1$. 
%Now consider the extensions
%\begin{equation*}
%1\to H \to \overline{\sigma}^{-1}(G) \stackrel{\pi}{\to} \overline{\sigma}^{-1}(G)/H\to 1,
%\end{equation*}
%where $H\cong \Z [G]$ and
%\begin{equation}\label{eq:br}
%1\to H \to \pi^{-1}(\iota_1(G))\to \iota_1(G)\to 1.
%\end{equation}
%Applying \repr{null} to the extension~\reqref{br}, it follows that $G$ embeds in $\pi^{-1}(\iota_1(G))$, therefore in $B_9/\Gamma_2(P_9)$.

\item\label{it:examb} We
% Here we show an example of a non Abelian group which is not of the form $\Z_n\rtimes \Z_m$.
now give an explicit example of a non-Abelian group $G$ of the form $A\rtimes \Z_{m}$ that embeds in $B_{\ord{A}}/\Gamma_2(P_{\ord{A}})$, where $A$ is a non-cyclic finite Abelian group. To our knowledge, this is the first explicit example of a finite group that embeds in such a quotient but that is not a semi-direct product of two cyclic groups. In particular, this subgroup does not satisfy the hypothesis of \reth{main} either. 
%, where $n$ is strictly less than the order of $G$, and that is not of the form $\Z_n\rtimes \Z_m$. 
Let
 %Here are some examples of finite groups and embeddings in the symmetric
%group $\sn$, possibly
 % for the best  possible integer $n$.
  %I am lost how to generalize the groups which are semi-direct product.
%Nevertheless here are some examples
 %that one may ask for the least integer $n$ such that $G$ embedds in
%$B_n/\Gamma_2(D)$.\\
 %We do not know if the group $G$ embedds in  $B_9/\Gamma_2(P_9)$. {\bf
%!!!D!!! Using \cite{GGO1} one should be %able to decide if there exist an
%embedding which projects to $\mathcal{S}_9$ to the permutation group
%defined above}
$G$ be the Heisenberg group $\bmod{p}$ of order $p^{3}$, where $p$ is an odd prime.
%'$G$ be the group of order $p^3$, where $p$ is a prime different from $2$, which is the
%Heisenberg group $\bmod{p}$ \comj{Or ?}. 
There exists an extension of the form:
\begin{equation}\label{eq:extHeis}
1\to \Z_p \to G\to \Z_p\oplus\Z_p\to 1,
\end{equation}
and a presentation of $G$ is given by:
\begin{equation}\label{eq:heis}
\setangr{a,b,c}{\text{$a^p=b^p=c^p=1$, $c=[a,b]$ and $[a,c]=[b,c]=1$}},
\end{equation}
%$\langle a, b, c \mid c=[a,b], \ a^p=b^p=c^p=1, \  [a,c]=[b, c]=1 \rangle$, %set of generators $\{a,b,c\}$ 
where $c$ is an element of $G$ emanating from a generator of the kernel $\Z_p$ of the extension~\reqref{extHeis}, and $a$ and $b$ are elements of $G$ that project to the generators of the summands of the quotient. This group is also isomorphic to $(\Z_p\oplus \Z_p)\rtimes_{\theta} \Z_p$, where the action $\map{\theta}{\Z_{p}}[\aut{\Z_p\oplus \Z_p}]$ is given by $\theta(1_p)=\left( 
\begin{smallmatrix}
1 & 1 \\
0 & 1
\end{smallmatrix}
\right)$. From now on, we assume that $p=3$.
%, and the relations $[a,b]=c, [a,c]=1, [b,c]=1$.
 % where $a, b$ are generators of the first and the second summand of $\Z_p+\Z_p$,
%respectively, and $c$ is a generator of the kernel $\Z_p$. 
Consider the map from $G$ to $\sn[9]$ given by sending $a$ to $\alpha=(1,4,7)(2,5,8)(3,6,9)$ and $b$ to $\beta=(1)(2)(3)(4,5,6)(7,9,8)$, so $c$ is sent to $\gamma=[\alpha,\beta]= (1,2,3)(4,5,6)(7,8,9)$. The relations of~\reqref{heis} hold for the elements $\alpha,\beta$ and $\gamma$, and since the only non-Abelian quotient of $G$ is $G$ itself, it follows that this map extends to an embedding of $G$ in $\sn[9]$. 
%\comj{??}, \emph{i.e.}\ 
%$\alpha^3=\beta^3=\gamma^3=1$, $[\alpha,
%\beta]=\gamma$ and $[\alpha, \gamma]=[\beta, \gamma]=1$ hold.   
%Among the elements of the group generated by
%$\alpha$, $\beta$ we have the element  \break
 % $\alpha^2\beta
%=(1,3,5,7,9,2,4,6,8)(1,5,6,4,8,9,7,2,3)=(1)(4)(7)(2,8,5)(3,9,6)$. 
%The short exact sequence~\reqref{sesgammak} with $n=9$ and $k=2$
%%\begin{equation*}
%%1\to P_9/\Gamma_2(P_9) \to B_9/\Gamma_2(P_9) \to \sn[9] \to 1
%%\end{equation*}
%induces the following short exact sequence:
%\begin{equation*}
%1\to P_9/\Gamma_2(P_9) \to \overline{\sigma}^{-1}(G)  \to G \to 1,
%\end{equation*}
%where the rank of the free Abelian group $P_9/\Gamma_2(P_9)$ is $36$ and the order of $G$ is $27$. 
Once more, one may check that $\mathcal{O}$ and $\mathcal{B}\setminus \mathcal{O}$ are the orbits arising from the action of $G$ on $\mathcal{B}$. 
% gives rise to two orbits, the orbit $\mathcal{O}$ of $A_{1,2}$ containing nine elements and given by~\reqref{orbA12}, and that of $A_{5,9}$, which contains $27$ elements. Note that the orbits are the same as those of the first example.
%It is clear that the action of $G$ on the basis of the quotient of~\reqref{sesB9} is not free. A direct calculation shows that the orbit of $A_{5,9}$ contains $27$ elements, and that of $A_{1,2}$ contains $9$ elements. 
%the orbit of the element $A_{i,j}$ where
%either i or j is not in the set $\{1,2,3\}$,   is a full orbit of lenght 27. So the 
%other orbit should be of lenght 9.
It remains to show that $G$ embeds in $\widetilde{H}$.
%to show there exists an embedding $\iota\colon\thinspace G\lhra B_9/\Gamma_2(P_9)$, we first show that $G$ embeds in $\overline{\sigma}^{-1}(G)/H$, where $H\cong \Z[G]$ is generated by the orbit of the element $A_{5,9}$.
% \comj{Same question as above: is $H$ normal in $B_9/\Gamma_2(P_9)$?}. 
Let $\map{\iota}{G}[\widetilde{H}]$ be the map defined by $\iota(\alpha)=\widehat{\alpha}$, $\iota(\beta)=\widehat{\beta}$ and $\iota(\gamma)=\widehat{\gamma}$, where:
%To see this, consider the subgroup of $\widetilde{H}$ generated by the following elements:
\begin{equation*}
\text{$\widehat{\alpha}=w' \sigma_2 \sigma_1^{-1} \sigma_5 \sigma_4^{-1}\sigma_8 \sigma_7^{-1}w'^{-1}$, $\widehat{\beta}= \sigma_5 \sigma_4^{-1} \sigma_7 \sigma_8^{-1}$ and $\widehat{\gamma}=\sigma_2 \sigma_1^{-1} \sigma_5 \sigma_4^{-1}\sigma_8 \sigma_7^{-1}$,}
\end{equation*}
and where $w'= \sigma_3 \sigma_2 \sigma_4\sigma_6\sigma_5 \sigma_4\sigma_3 \sigma_7\sigma_6$. Using the notation of~\cite[equations~(14) and~(16)]{GGO1}, $\widehat{\alpha}=w' \delta_{0,3}\delta_{3,3} \delta_{6,3} w'^{-1}=w'\delta(0,3,3,3) w'^{-1}$, $\widehat{\beta}=\delta_{3,3} \delta_{6,3}^{-1}$ and $\widehat{\gamma}=w' \widehat{\alpha} w'^{-1}$. So these three elements are of order $3$ in $B_{9}/\Gamma_{2}(P_{9})$ by the argument of~\cite[line~4, p.~412]{GGO1}, and hence they satisfy the first three relations of~\reqref{heis} in $\widetilde{H}$, $a$, $b$ and $c$ being taken to be $\widehat{\alpha}, \widehat{\beta}$ and $\widehat{\gamma}$ respectively.
%\begin{equation*}\widehat{\alpha}^3=\widehat{\beta}^3=\widehat{\gamma}^3=1, \  \widehat{\gamma}=[\widehat{\alpha} ,\widehat{\beta} ], \ [\widehat{\alpha}, \widehat{\gamma}]=
%[\widehat{\beta}, \widehat{\gamma}]=1.
%\end{equation*}
%\comj{this added:} 
One may check in a straightforward manner that $[\widehat{\beta}, \widehat{\gamma}]=1$ as elements of $B_{9}$, hence $[\widehat{\beta}, \widehat{\gamma}]=1$ in $\widetilde{H}$. To see that the two remaining relations of~\reqref{heis} hold, as in the first example, we use the method of crossing numbers of the strings given in~\cite[Proposition~15]{GGO1}, but in $\widetilde{H}$ rather than $P_9/\Gamma_2(P_9)$. In this way, we see from Figures~(\ref{fig:fig4})(\subref{fig:fig4a}) and~(\subref{fig:fig4b}) that 
%\comj{Add some pictures?}\como{We may put here pictures like those given in Figures \ref{fig:braid_1_1} and \ref{fig:braid_1_2}.} we may show that 
$[\widehat{\alpha}, \widehat{\beta}] \ldotp \widehat{\gamma}^{-1}=1$ and $[\widehat{\alpha}, \widehat{\gamma}]=1$
%and $[\widehat{\beta}, \widehat{\gamma}]=1$ 
in $\widetilde{H}$. It thus follows that $\langle\, \widehat{\alpha}, \widehat{\beta}, \widehat{\gamma} \,\rangle$ is a quotient of $G$, but since this subgroup is non-Abelian, and the only non-Abelian quotient of $G$ is itself, we conclude that $\langle\, \widehat{\alpha}, \widehat{\beta} \,\rangle \cong G$, and hence $\iota$ is an embedding. Once more, it follows from the discussion at the beginning of these examples that $G$ embeds in $\pi^{-1}(\iota(G))$, and therefore in $B_9/\Gamma_2(P_9)$.
%The first three relations are obtained using~\cite[Proposition 21]{GGO1}, and the other three relations may be checked using the notion of crossing number~\cite[Proposition 15]{GGO1}.  
%%\item Let $G=(\Z_9\oplus \Z_9)\rtimes \Z_3$.    One can try to embed this group
%%in $\sn[18]$. Not clear if $18$  is the best possible integer that one can
%%embed  $(\Z_9\oplus\Z_9)$.
%%Find an embedding of $G$ in   $\sn[18]$ \como{We may try to use the
%%software GAP to do it.}.
%The fact that the embedding of $G$ into  $\overline{\sigma}^{-1}(G)/H$ lifts to an embedding $\iota\colon\thinspace G\lhra B_9/\Gamma_2(P_9)$ follows using the same argument as in part~(\ref{it:exama}). 
%Observe that this group $G$ is also isomorphic to $(\Z_p\oplus \Z_p)\rtimes_{\theta} \Z_p$ where $\theta(1_p)=\left(
%\begin{smallmatrix}
%1 & 1 \\
%0 & 1
%\end{smallmatrix}
%\right)$. It would be nice to classify those finite groups of the form $A\rtimes_{\theta} \Z_m$, where $A$ is a finite abelian group, that embed in $B_{\ord{A}}/\Gamma_2(P_{\ord{A}})$.
\end{enumerate}
\end{exms}
 
\begin{rems}\mbox{}
\begin{enumerate}
\item Let $G$ be one of the two groups of order $27$ analysed in \rexs{groups27}. With the notation introduced at the beginning of \resec{semidirect}, the fact that $G$ embeds in $B_{9}/\Gamma_{2}(P_{9})$ implies that $\ell_2(G)\leq 9$. On the other hand, if $G$ embeds in $\sn[r]$ then $r\geq 9$ by Lagrange's Theorem. Hence $m(G)\geq 9$, and it follows from~\reqref{ineqmG} that $m(G)=\ell_2(G)=9$, which is coherent with~\cite[Corollary~13]{MaV} in the case $k=2$.

\item The finite groups of the form $A\rtimes_{\theta} H$, where $A$ is a finite Abelian group, $H$ is an arbitrary finite group, and $\map{\theta}{H}[\aut{A}]$ is injective, embed in $S_{A}$ by \relem{Ginjphi}.
% {\bf the new remark we added about Lemma \ref{lem:fundamental} }. 
From~\cite[Corollary~13]{MaV}, if the order of $A\rtimes_{\theta} H$ is odd then it  
 embeds in $B_{\ord{A}}/\Gamma_2(P_{\ord{A}})$. 
 %\comj{In the case $k=2$, can we say something here using~\cite[Corollary~13]{MaV}}.
% \comj{Perhaps put this first sentence at the beginning of Example~\ref{it:examb}?} 
 We would like to be able to determine which of these groups 
 %the finite groups of the form $A\rtimes_{\theta} \Z_m$, where $A$ is a finite Abelian group, that 
 embed in $B_{\ord{A}}/\Gamma_3(P_{\ord{A}})$.
 % \comj{In the case $k=2$, can we say something here using~\cite[Corollary~13]{MaV}}.
\end{enumerate}
\end{rems}
 
%For $G$ any of the  two examples above,  since the order is $27$ by Lagrange Theorem we have $m(G)\geq 9$. From the embbeding above $\ell_2(G)\leq 9$. So follows   that $m(G)=\ell_2(G)=9$. {\bf !!!D!!!???????}

\end{document}